\documentclass[12pt]{amsart}

\usepackage{amssymb,amsmath}
\usepackage{shuffle}
\usepackage{gensymb}
\usepackage{latexsym}
\usepackage{tikz}
\usepackage[utf8]{inputenc}
\usepackage{hyperref}
\usepackage{bm}
\usepackage[all]{xy}
\usepackage{todonotes}
\usepackage{cleveref}

\newtheorem{example}{Example}[section]

\newtheorem{remark}[example]{Remark}
\newtheorem{theorem}[example]{Theorem}

\newtheorem{corollary}[example]{Corollary}

\newtheorem{definition}[example]{Definition}
\newtheorem{proposition}[example]{Proposition}

\newtheorem{lemma}[example]{Lemma}
\newtheorem{OpenProblem}[example]{Open problem}

\renewcommand{\leq}{\leqslant}
\renewcommand{\le}{\leqslant}
\renewcommand{\Im}{\text{Im}}
\def\la{\lambda}
\def\si{\sigma}
\def\bd{\gamma}
\def\ex{\text{ex}}
\def\Id{\text{Id}}
\def\dimKK{k}

\def\U{U}

\def\X{{\mathbb X}}
\def\Y{{\mathbb Y}}

\def\Sym{{\bf Sym}}
\def\QSym{{\it QSym}}

\def\WQSym{{\bf WQSym}}

\def\eval{{\rm eval}}

\def\<{\langle}
\def\>{\rangle}

\def\FF{{\mathcal F}}
\def\C{{\mathbb C}}

\def\N{{\mathbb N}}

\def\SG{{\mathfrak S}}
\def\H{{\sf H}}

\def\A{{\mathbb A}}

\def\KK{{\mathcal K}}

\def\y{{\bf y}}
\def\x{{\bf x}}
\def\bc{{\bf c}}
\def\bv{{\bf v}}
\def\bw{{\bf w}}
\def\pp{{\bm p}}
\def\qq{{\bm q}}

\def\pack{{\rm pack}}

\def\Sol{ \mathcal{S}} 
\def\SolNC{\mathcal{S}_{\text{nc}}} 
\def\IC{\mathcal{IC}} 
\def\MRCoord{\mathcal{MC}} 
\def\Young{\mathcal{Y}} 
\DeclareMathOperator{\ActY}{eval_\Young} 
\def\QLa{Q\Lambda} 
\DeclareMathOperator\Ker{Ker}
\DeclareMathOperator\Ch{Ch}
\newdimen\squaresize
\newdimen\thickness         

\newcommand{\pref}[1]{(\ref{#1})}
\def\ie{{\em i.e. }}
\def\Phixp{\Phi_{x \to p,q}}
\def\Phipx{\Phi_{p,q \to x}}
\def\Phiab{\Phi_{a \to b,d}}
\def\Phiba{\Phi_{b,d \to a}}
\def\square#1{\hbox{\vrule width \thickness
   \vbox to \squaresize{\hrule height \thickness\vss                            
      \hbox to \squaresize{\hss#1\hss}
   \vss\hrule height\thickness} 
\unskip\vrule width \thickness} 
\kern-\thickness}                                                            
                               
\def\vsquare#1{\vbox{\square{$\casestyle#1$}}\kern-\thickness}

\def\young#1{\vcenter{%
    \squaresize=18pt\thickness=0.5pt\let\casestyle=\relax
    \vbox{\smallskip\offinterlineskip
      \halign{&\vsquare{##}\cr #1}}}}

\def\moyyoung#1{\vcenter{%
    \squaresize=28pt\thickness=0.6pt\let\casestyle=\relax
    \vbox{\smallskip\offinterlineskip
      \halign{&\vsquare{##}\cr #1}}}}

\def\bigyoung#1{\vcenter{%
    \squaresize=38pt\thickness=0.8pt\let\casestyle=\displaystyle
    \vbox{\smallskip\offinterlineskip
      \halign{&\vsquare{##}\cr #1}}}}

\def\boxit#1#2{\setbox1=\hbox{\kern#1{#2}\kern#1}%
\dimen1=\ht1 \advance\dimen1 by #1 \dimen2=\dp1 \advance\dimen2 by #1
\setbox1=\hbox{\vrule height\dimen1 depth\dimen2\box1\vrule}%
\setbox1=\vbox{\hrule\box1\hrule}%
\advance\dimen1 by .4pt \ht1=\dimen1
\advance\dimen2 by .4pt \dp1=\dimen2 \box1\relax}

\def\gf#1#2{\genfrac{}{}{0pt}{}{#1}{#2}}

\title[Quasi-symmetric functions as polynomials on Young diagrams]%
{Quasi-symmetric functions\\ as  polynomial functions on Young diagrams}

\author[J.-C. Aval, V. F\'eray, J.-C. Novelli, J.-Y. Thibon]%
{Jean-Christophe Aval, Valentin F\'eray, Jean-Christophe Novelli,\\ and
Jean-Yves Thibon}

\address[Aval]{LaBRI, Universit\'e Bordeaux I \\
351 cours de la lib\'eration \\ 33405 Talence cedex \\
France}

\address[Féray]{Institut für Mathematik, Universität Zürich, Winterthurerstrasse 190,
8057 Zürich, Switzerland
}

\address[Novelli and Thibon]{Institut Gaspard Monge, Universit\'e Paris-Est
Marne-la-Vall\'ee \\
5 Boulevard Descartes \\Champs-sur-Marne \\77454 Marne-la-Vall\'ee cedex 2 \\
France\bigskip}

\email[Jean-Christophe Aval]{aval@labri.fr}
\email[Valentin F\'eray]{valentin.feray@math.uzh.ch}
\email[Jean-Christophe Novelli]{novelli@univ-mlv.fr}
\email[Jean-Yves Thibon]{jyt@univ-mlv.fr}

\date{}

\begin{document}

\begin{abstract}
We determine the most general form of a smooth function on Young diagrams,
that is, a polynomial in the interlacing or multirectangular coordinates whose
value depends only on the shape of the diagram. We prove that the algebra of
such functions is isomorphic to quasi-symmetric functions, and give a
noncommutative analog of this result.
\end{abstract}



\maketitle

\section{Introduction}

A central question in this paper is the following problem:
\begin{quote}
Characterize the polynomials $f(x_1,x_2,x_3,\ldots)$ in infinitely many
variables\footnote{Understood as elements of an inverse limit, see
Section~\ref{SectDefNot}.}
such that 
\begin{equation}
\label{eqfoncQS}
f(x_1,x_2,\dots)|_{x_i=x_{i+1}} =
f(x_1,\dots,x_{i-1},x_{i+2},\dots).
\end{equation}
\end{quote}

\subsection{Motivation: Young diagrams and Equation \eqref{eqfoncQS}}
\label{SubsectMotivation}

Consider a Young diagram $\lambda$ drawn with the Russian convention,
(\emph{i.e.}, draw it with the French convention, rotate it counterclockwise
by $45\degree$ and scale it by a factor $\sqrt{2}$).
Its border can be interpreted as the graph of a piecewise affine function.
We denote by $x_1, x_2,\dots,x_{2m+1}$ the abscissas of its local minima and
maxima in decreasing order, see Figure~\ref{FigRussian}.

\begin{figure}[ht]
    \begin{tikzpicture}
      \begin{scope}[draw=black,scale=.5]
          \draw[->,thick] (0,0) -- (6,0);
          \foreach \x in {1, 2, 3, 4, 5}
              { \draw (\x, -2pt) node[anchor=north] {{\tiny{$\x$}}} -- (\x,
2pt); }
          \draw[->,thick] (0,0) -- (0,5);
          \foreach \y in {1, 2, 3, 4}
              { \draw (-2pt,\y) node[anchor=east] {{\tiny{$\y$}}} -- (2pt,\y); }
          \draw[ultra thick,draw=black] (4,0) -- (4,2) -- (2,2) -- (2,3) -- (0,3) ;
          \fill[fill=gray,opacity=0.1] (4,0) -- (4,2) -- (2,2) -- (2,3) -- (0,3) -- (0,0) -- cycle ;
      \end{scope}
\begin{scope}[xshift=7cm, yshift=-0cm, scale=0.5]
       \begin{scope}
          \clip (-4.5,0) rectangle (5.5,5.5);
          \draw[thin, dotted] (-6,0) grid (6,6);
          \begin{scope}[rotate=45,draw=gray,scale=sqrt(2)]
              \clip (0,0) rectangle (4.5,5.5);
              \draw[thin, dotted] (0,0) grid (6,6);
          \end{scope}
      \end{scope}

      \draw[->,thick] (-6,0) -- (6,0) node[anchor=west] {\tiny{$x$}};
      \draw[->,thick] (0,-0.4) -- (0,6);
%
%
\begin{scope}[draw=gray,rotate=45,scale=sqrt(2)]
          \draw[->,thick] (0,0) -- (6,0);
          \draw[->,thick] (0,0) -- (0,5);
          \draw[ultra thick,draw=black] (5.5,0) -- (4,0)  -- (4,2) -- (2,2) -- (2,3) --
          (0,3) -- (0,4.5) ;
          \fill[fill=gray,opacity=0.1] (4,0) -- (4,2) -- (2,2) -- (2,3) -- (0,3) -- (0,0) -- cycle ;
      \end{scope}
 \draw[ultra thick,dotted] (4,4) -- (4,0) node[anchor=north] {\tiny{$x_1=4$}};
 \draw[ultra thick,dotted] (2,6) -- (2,-.8) node[anchor=north] {\tiny{$x_2=2$}};
 \draw[ultra thick,dotted] (0,4) -- (0,0) node[anchor=north] {\tiny{$x_3=0$}};
 \draw[ultra thick,dotted] (-1,5) -- (-1,-.8) node[anchor=north] {\tiny{$x_4=-1$}};
 \draw[ultra thick,dotted] (-3,3) -- (-3,0) node[anchor=north] {\tiny{$x_5=-3$}};
\end{scope}
    \end{tikzpicture}
    \caption{Young diagram $\lambda=(4,4,2)$
    and the graph of the associated function $\omega_\lambda$.
}
    \label{FigRussian}
\end{figure}
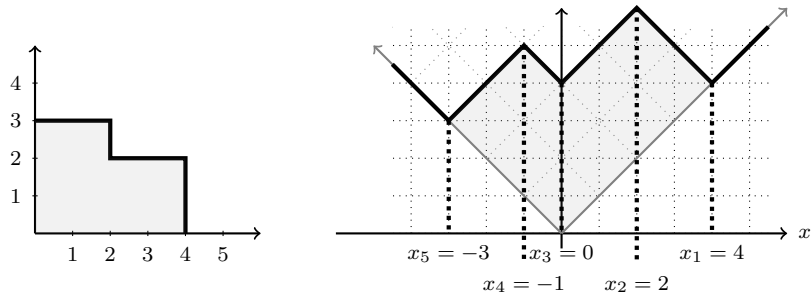

These numbers $x_1,x_2,\cdots,x_{2m+1}$
are called {\em (Kerov) interlacing coordinates}, see, {\it e.g.},
\cite[Section 6 with $\theta=1$]{OlshanskiPlancherelAverage}.
They are usually labeled with two different alphabets for minima
and maxima, but we shall rather use the same alphabet here and distinguish
between odd-indexed and even-indexed variables when necessary.

Note that not any decreasing sequence of integers can be obtained in this
way, as interlacing coordinates always satisfy the relation
$\sum_i (-1)^i x_i=0$.
Yet, this construction defines an injective map
\begin{equation}
\IC: \big\{\text{Young diagrams}\big\} \to \left\{\begin{tabular}{c}
    finite sequences\\
    of integers
\end{tabular}\right\}.
\end{equation}

A polynomial in infinitely many variables can be evaluated on any finite
sequence. By composition with $\IC$, one may wish to interpret it as a function
on all Young diagrams.

A Young diagram can be easily recovered from its Kerov coordinates
$x_1, \dots,x_{2m+1}$.
To obtain its border, first draw the half-line $y=-x$ for $x\leq x_{2m+1}$,
then, without raising the pen, draw line segments of slope alternatively $+1$
and $-1$ between points of $x$-coordinates $x_{2m+1},x_{2m},\dots,x_1$ and
finally a half-line of slope $+1$ for $x \geq x_1$.
Starting with a decreasing integral sequence satisfying $\sum_i (-1)^i x_i=0$,
the last half-line has equation $y=x$ and the resulting broken line  can be
interpreted as the border of a Young diagram drawn with the Russian
convention (see \cite[Proposition 2.4]{IO02}).

Apply now the same process to a non-increasing sequence 
$x_1, x_2,\dots,x_{2m+1}$ such that $x_i=x_{i+1}$.
Reaching the $x$-coordinate $x_i=x_{i+1}$, one has to change twice the sign of
the slope, that is, to do nothing. Hence, one obtains the same diagram as for
sequence
\[x_1,\cdots,x_{i-1},x_{i+2},\cdots,x_{2m+1}.\]
Therefore, if one wants to interpret a polynomial in infinitely many variables
as a function of Young diagrams, it is natural to require that it satisfies
Equation~\eqref{eqfoncQS}.

\subsection{Solution}

In Section \ref{SectSolution}, we shall describe the algebraic structure of 
the space $\Sol$ of solutions of Equation \eqref{eqfoncQS}.

\begin{theorem}
\label{ThmSolEqQSym}
As an algebra, $\Sol$ is isomorphic to $\QSym$, the algebra of {\em
quasi-symmetric functions}.
\end{theorem}

The algebra of quasi-symmetric functions is a natural extension of that of
symmetric functions, widely studied in the literature.
Its definition is recalled in Section \ref{SectDefNot}.
The isomorphism of Theorem \ref{ThmSolEqQSym} is naturally described in terms
of Hopf algebra calculus: the solutions to Equation
\eqref{eqfoncQS} are the quasi-symmetric functions evaluated on the virtual
alphabet (this notion is defined in Section~\ref{SubsectQSymAlph})
\begin{equation}
\label{EqVirtualAlphabet}
   \X =\ominus (x_1) \oplus (x_2) \ominus (x_3) \oplus (x_4) \cdots 
\end{equation} 
We emphasize here the similarity with a result of J.R. Stembridge
\cite{StembridgeSuperSym}.
He studied the solutions of Equation~\eqref{eqfoncQS} which are in addition
symmetric in the odd-indexed variables $x_1,x_3,\dots$ and separately in the
even-indexed variables $x_2,x_4,\dots$.
He proved that the space of these functions is algebraically generated
by the power sums in the virtual alphabet above, that is
\begin{equation}\label{EqPowerSum}
    p_k(\X) = \sum_i (-1)^i x_i^k.
\end{equation}
In other words, the {\em symmetric} solutions to \eqref{eqfoncQS}
are the symmetric functions evaluated on $\X$.
From this point of view, our result is the natural quasi-symmetric analogue
of Stembridge's theorem.

\subsection{Back to Young diagrams}

As explained in Section~\ref{SubsectMotivation}, the solutions of
\eqref{eqfoncQS} can be interpreted as functions on Young diagrams.

It turns out that symmetric polynomials evaluated on $\X$ (which form a subset
of $\Sol$) correspond to a well-known algebra of functions on Young diagrams,
denoted here by $\Lambda$, and referred to as {\em symmetric functions on
Young diagrams}%
\footnote{The usual terminology is {\em polynomial functions on Young
diagrams}, which can be confusing as some functions that do not belong to this
algebra depend polynomially on interlacing coordinates $x_1,x_2,\ldots$.\label{footnote:symmetric}}.
This algebra, introduced by Kerov and Olshanski \cite{KO94}, is algebraically
generated by the family of functions~\eqref{EqPowerSum}
\cite[Corollary 2.8]{IO02}.
Therefore, Equation \eqref{eqfoncQS} leads us to a new algebra of functions on
Young diagrams, strictly larger than that of Kerov and Olshanski, see Section
\ref{SubsectActY} for details.

Our algebra $\Sol$ has a rich structure (quasi-symmetric functions), and it
contains some non-{\em symmetric} functions, denoted by $N_G$, which have
played an important role in the approach to representation theory of the
symmetric group developed in some recent papers (see, {\em e.g.}, \cite{DFS10}
and references therein).
In Section \ref{SubsectNG}, we give a new formula for $N_G$ in terms of
quasi-symmetric functions of $\X$.
\medskip

Besides, in \cref{SubsectMultirec,SubsectLinkTwoSystems},
we describe our new algebra $\Sol$ in terms 
of the so-called multirectangular coordinates of Young diagrams.
We show that
the algebra $\Sol$ is also the set of polynomials in multirectangular coordinates,
which can be interpreted as function of Young diagrams 
({\em i.e.} which takes the same values on different sets of multirectangular coordinates
of a Young diagram).
Interestingly enough, looking at this multirectangular coordinates yields
a two-alphabet version of a basis of $\QSym$ introduced
by Malvenuto and Reutenauer \cite{MR}, whose product is given by the shuffle
operation on parts of the composition -- see \cref{SubsectH}.
\medskip

Other sets of coordinates of Young diagrams have turned out to be useful,
in particular in the context of Kerov and Olshanski algebra of symmetric
functions on Young diagrams:
the row coordinates,
the (modified) Frobenius coordinates and the multiset of contents of boxes.
It would perhaps be fruitful to investigate our new algebra in terms of these sets
of parameters.
We discuss this as a direction for future research in Section \ref{SectFuture}.
\subsection{Generalization to a noncommutative framework}

To avoid confusion, we will always use a set of variables $\{a_1,a_2,\dots\}$ 
for polynomials in noncommuting variables.
A functional equation, analogous to \eqref{eqfoncQS}, can be considered:
\begin{quote}
Characterize the polynomials $P(a_1,a_2,a_3,\ldots)$ in infinitely many
{\em non commuting} variables such that 
\begin{equation}
\label{eqfoncWQS}
P(a_1,a_2,\dots)|_{a_i=a_{i+1}} =
P(a_1,\dots,a_{i-1},a_{i+2},\dots).
\end{equation}
\end{quote}

We solve this problem in Section \ref{SectNonCommutative}.
\begin{theorem}
    As an algebra, the space of solutions of \eqref{eqfoncWQS}
    is isomorphic to the algebra of 
    {\em word quasi-symmetric functions}  $\WQSym$.
    \label{ThmSolEqWQSym}
\end{theorem}
The definition of $\WQSym$ is recalled in Section \ref{SubsectDefWQSym}.
This algebra is the natural noncommutative analogue of $\QSym$.

As in the commutative setting, the solutions are constructed from the elements
of $\WQSym$ using a virtual alphabet:
\begin{equation}\label{EqVirtualAlphabetNC}
\A = \ominus (a_1) \oplus (a_2) \ominus (a_3) \oplus (a_4) \ominus \dots
\end{equation}
However, there was here an extra difficulty. Differences of alphabets for
word quasi-symmetric functions cannot be defined by means of the antipode,
which is not involutive, so we had to introduce an ad-hoc definition, see
Section~\ref{SubsectEvalNCAlph}.

Evaluating these functions in noncommuting variables on the interlacing
coordinates of Young diagrams does not bring any new information, as this
operation factors through the commutative version.
It is however interesting to study the change of variables between interlacing
coordinates and multirectangular coordinates in the noncommutative framework.
We study this morphism and describe its kernel in
Section~\ref{SectNonCommutativeMultirect}.
This result involves the lifting of a basis of $\QSym$ introduced by K. Luoto
(see Remark \ref{rmq:Luoto})
and the computation of the dimension of the smallest two-sided ideal of $\WQSym$
containing the element of degree $1$ and stable by the actions of the symmetric groups
(see Theorem \ref{thm:Ker_NC}, first item),
which might be of interest on their own.



\section{Definitions and notations in the commutative framework}
\label{SectDefNot}

\subsection{Stable polynomials}

By ``polynomial in infinitely many variables'', we mean an element of an
inverse limit in the category of graded rings, {\it i.e.}, a homogeneous
polynomial of degree $d$ is a sequence $R=(R_n(x_1,\ldots,x_n))_{n\ge 0}$ of
homogeneous polynomials of degree $d$ such that
$R_{n+1}(x_1,\ldots,x_n,0)=R_n(x_1,\ldots,x_n)$.
These objects are sometimes called {\em stable polynomials}.
Their set will be denoted $\C[X]$, where $X$ is the infinite variable set
$X=\{x_1,x_2,\dots\}$ (which should not be confused with the virtual alphabet
$\X$ defined by Equation \eqref{EqVirtualAlphabet}).

This kind of construction is classical in algebraic combinatorics:
for instance, symmetric functions (see \cite{Mcd}) and quasi-symmetric
functions (see below) are built in this way.
\medskip

In this context, Equation \eqref{eqfoncQS} should be understood as follows.
A stable polynomial $f=(f_n)_{n\ge 0}$ is solution of \eqref{eqfoncQS} if for each
$n \geq 2$ and each $1\le i<n$, one has:
\[f_n(x_1,\dots,x_n)_{|x_{i+1}=x_i}
  =f_{n-2}(x_1,\dots,x_{i-1},x_{i+2},\ldots,x_n).\]
The left-hand side means that we substitute $x_{i+1}$ by $x_i$.
Then the equality must be understood as an equality between polynomials in
$x_1,\dots,x_{i-1},x_{i+2},\dots,x_n$.
In particular, the left-hand side must be independent of $x_i$.
\medskip

We will also need to consider polynomials in two infinite sets of variables.
By definition, an element of $\C[\bm{p},\bm{q}]$ is a sequence
$(h_m)_{m \ge 0}$, where each $h_m$ is a polynomial in the $2m$ variables
$p_1,\dots,p_m,q_1,\dots,q_m$
satisfying the stability property
\begin{equation}
h_{m+1}
\left( \begin{array}{cccc}
    p_1 & \dots & p_m & 0\\
    q_1 & \dots & q_m & 0
\end{array} \right) = 
h_m
\left( \begin{array}{ccc}
    p_1 & \dots & p_m \\
    q_1 & \dots & q_m
\end{array} \right).
\end{equation}

Finally, we shall sometimes define stable polynomials by a sequence of polynomials in
an odd number of variables $(R_{2m+1})_{m \ge 0}$ such that
\[R_{2m+1}(x_1,\dots,x_{2m-1},0,0)=R_{2m-1}(x_1,\dots,x_{2m-1}).\]
This is not an issue, as such a sequence can be extended in a unique way to a
stable sequence $(R_n)_{n \ge 0}$ by setting
\[R_{2m}(x_1,\dots,x_{2m}) = R_{2m+1}(x_1,\dots,x_{2m},0).\]

\subsection{The Hopf algebra of quasi-symmetric functions}
\label{SubsectDefQSym}

Quasi-symmetric functions were introduced by I.~Gessel \cite{Ges} 
and may be seen as a generalization of the notion of symmetric functions.

A {\em composition} of $n$ is a sequence $I=(i_1,i_2,\dots,i_r)$
of positive integers, whose sum is equal to $n$.
The notation $I\vDash n$ means that $I$ is a composition of $n$
and $\ell(I)$ denotes the number of parts of $I$.
In numerical examples, it is customary to omit the parentheses and the commas.
For example, $212$ is a composition of $5$.
Given two compositions $I=(i_1,i_2,\dots,i_r)$ and $J=(j_1,j_2,\dots,j_s)$
their {\em concatenation} is $I \cdot J=(i_1,\dots,i_r,j_1,\dots,j_s)$.

Recall that the {\em multiset} of {\em shuffles} of $I$ and $J$ is defined
recursively by:
\begin{equation}
I\shuffle J=(i_1)\cdot \Big((i_2,\dots,i_r)\shuffle J \Big)
\sqcup (j_1)\cdot \Big(I \shuffle (j_2,\dots,j_s)\Big).
\end{equation}

In quasi-symmetric function theory, we use a slight modification of the
shuffle, called {\em quasi-shuffle}, defined recursively by:
\begin{equation}
\begin{split}
I\star J = (i_1)\cdot \Big((i_2,\dots,i_r)\star J \Big)
           & \sqcup (j_1)\cdot \Big(I \star (j_2,\dots,j_s)\Big) \\
           & \sqcup (i_1+j_1) \Big( (i_2,\dots,i_r)\star(j_2,\dots,j_s)\Big).
\end{split}
\end{equation}

For a composition $I=(i_1,\dots,i_r)$, we denote by $\bar I$ the composition
$(i_r,\dots,i_1)$ {\em mirror} to $I$.
For two compositions $I=(i_1,\dots,i_r)$ and $J=(j_1,\dots,j_s)$
we say that $J$ is a {\em refinement} of $I$ if 
$$\{i_1,i_1+i_2,\dots,i_1+\cdots+i_r\}\subseteq
  \{j_1,j_1+j_2,\dots,j_1+\cdots+j_s\},$$
which will be denoted by $I\le J$.

Consider the algebra
$\C[X]$ of polynomials in the totally ordered commutative alphabet
$X=\{x_1,x_2,\dots\}.$
Monomials $X^\bv:=x_1^{v_1} x_2^{v_2} \dots$
correspond to vectors $\bv=v_1,v_2,\dots$ with finitely many
non-zero entries.
For such a vector, we denote by $\bv_{\leftarrow}$ the vector obtained by
omitting the zero entries.
\begin{definition}
A polynomial $P\in\C[X]$ is said to be {\em quasi-symmetric} if and only if
for any $\bv$ and $\bw$ such that $\bv_{\leftarrow}=\bw_{\leftarrow}$,
the coefficients of $X^\bv$ and $X^\bw$ in $P$ are equal.

One can easily prove that the set of quasi-symmetric polynomials           
is a subalgebra of $\C[X]$, called quasi-symmetric function ring and denoted
$\QSym$.
\end{definition}
It should be clear that any symmetric polynomial is quasi-symmetric.
The algebra $\QSym$ of quasi-symmetric functions has a basis of monomial
quasi-symmetric functions $M$ indexed by compositions $I= (i_1,\dots, i_r)$,
where
\begin{equation}\label{eq:QSymM}
M_I = \sum_{a_1<\cdots<a_r} x_{a_1}^{i_1} \cdots x_{a_r}^{i_r} .
\end{equation}
By convention, $M_{()}=1$, where $()$ designs the empty composition.

The product of monomial functions is given by the quasi-shuffle of their
indices: $M_I\,M_J=\sum_{K\in I\star J}M_K.$
For instance,
$$M_{2}\,M_{11} = M_{112} + M_{121} + M_{211} + M_{13} + M_{31} .$$
This given, the set $\QSym$ is an algebra with unit $M_0=1$.
Moreover it is graded by the usual degree.
The coproduct of $\QSym$ may be defined on monomial functions through the
deconcatenation of compositions:
\begin{equation}\label{eq:coproduct}
\Delta(M_I)=\sum_{k=0}^r M_{(i_1,\dots,i_k)}\otimes M_{(i_{k+1},\dots,i_r)}.
\end{equation}
As an example, one has
$$\Delta M_{21} = {1}\otimes M_{21}+{M_2}\otimes M_{1}+M_{21}\otimes {1}.$$ 
This operation endows $\QSym$ with a bialgebra structure, whence a Hopf
algebra structure since it is graded. The antipode, denoted as usual by $S$
has been explicitly computed by C.~Malvenuto \cite{malvenuto} and R.~Ehrenborg
\cite{ehrenborg}:
\begin{equation}\label{eq:antipode}
S(M_I)={(-1)}^{\ell(I)}\,\sum_{J\le I} M_{\bar J}.
\end{equation}
For example, we have:
$$S(M_{122})=-(M_{221}+M_{41}+M_{23}+M_{5}).$$

\subsection{Evaluation of quasi-symmetric functions on sums and differences of alphabets}
\label{SubsectQSymAlph}

We explain now how to evaluate quasi-symmetric functions on sums or differences of alphabets.
To start with, we shall give another interpretation of the coproduct $\Delta$.
To do this, we consider two ordered alphabets $X$ and $Y$ and we denote by
$X \oplus Y$ their ordinal sum, that is, their disjoint union seen as an
ordered alphabet with $x<y$ for $x\in X$ and $y\in Y$.
From \eqref{eq:coproduct}, we may check that
\begin{equation} 
\Delta P=\sum_k F_k \otimes G_k
\end{equation}
implies 
\begin{equation}\label{EqSum}
    P(X \oplus Y)=\sum_k F_k(X) \, G_k(Y).
\end{equation} 
This defines sums of alphabets and the evaluation of quasi-symmetric functions
on these.

Let us now introduce the formal inverse of $Y$ for operation $\oplus$,
denoted by $\ominus Y$.
Of course, $\ominus Y$ does not exist as an alphabet;
we refer to it as a {\em virtual alphabet}.
We may write 
\begin{equation}\label{eq:diff}
P(X \ominus Y)=\sum_k F_k(X) \, G_k( \ominus Y)
\end{equation}
with the same notations as above.
For this to make sense, we need to define the evaluation $P(\ominus Y)$ of a
quasi-symmetric function $P$ on the opposite of $Y$.
We set 
\begin{equation}\label{eq:virtual}
    P(\ominus Y)=S(P)(Y). 
\end{equation}
Let us explain this choice.
Using the axiom of the antipode in a Hopf algebra,
we observe that \eqref{eq:diff} evaluated at $X=Y$ gives,
for any homogeneous quasi-symmetric function $P$ of positive degree,
\begin{align}
    P(Y \ominus Y)&=0. \label{EqYMoinsY}
\intertext{Similarly,}
P(\ominus Y \oplus Y)&=0. \label{EqMoinsYPlusY}
\end{align}
This is consistent with the fact that $\ominus Y$ is the inverse of $Y$
for $\oplus$, that is that $Y \ominus Y=\ominus Y \oplus Y= \emptyset$ 
($\emptyset$ is here the empty alphabet).

Equations \eqref{EqSum} and \eqref{eq:virtual} enable us to evaluate
quasi-symmetric functions on differences of alphabets.  As an example, we
have:
\begin{align*}
    M_{21}(X \ominus Y)&=M_{21}(X) + M_2(X) S(M_1)(Y) +S(M_{21})(Y)\\
    &=M_{21}(X)-M_2(X)M_1(Y)+M_{12}(Y)+M_{3}(Y).
\end{align*}
It is also possible to have more complicated linear combinations of alphabets
(beware that $\oplus$ is not a commutative operator), for instance,
\begin{align*}                                       
    M_{21}(X \ominus Y \oplus Z)
    &=M_{21}(X) + M_2(X) S(M_1)(Y) + M_2(X) M_1(Z) \\
    &\qquad + S(M_{21})(Y) +
    S(M_2)(Y) M_1(Z) + M_{21}(Z) ;\\
    &= M_{21}(X) - M_2(X) M_1(Y) + M_2(X) M_1(Z) \\
    &\qquad +M_{12}(Y)+M_{3}(Y) 
        - M_2(Y) M_1(Z) + M_{21}(Z).
\end{align*}

\section{Solution of the problem}\label{SectSolution}

Consider the virtual ordered alphabet for quasi-symmetric functions
\begin{equation}
   \X =\ominus (x_1) \oplus (x_2) \ominus (x_3) \oplus (x_4) \cdots 
\end{equation}
If $F$ is a quasi-symmetric function,
one can define a stable polynomial $F(\X)$ as follows: for $n \ge 0$,
\[(F(\X))_n (x_1,\dots,x_n) = 
F \big(\ominus (x_1) \oplus (x_2) \ominus \dots (x_n) \big).\]

Using Equations \eqref{EqYMoinsY} and \eqref{EqMoinsYPlusY},
we see that setting $x_i=x_{i+1}$ in this alphabet cancels
these variables, so that $F(\X)$ satisfies (\ref{eqfoncQS}). The converse is
also true:

\begin{theorem}
A function $f$ satisfies the functional equation \eqref{eqfoncQS}
if and only if $f\in\QSym(\X)$.
\end{theorem}

\begin{proof}
Note that a polynomial $f$ satisfies Equation \eqref{eqfoncQS}
if and only if all its homogeneous components do.
Therefore it is enough to prove the statement for a homogeneous function $f$.

Let us first prove that the dimension of the space of homogeneous polynomials
in $\C[X]$ of degree $n$ satisfying \pref{eqfoncQS} is at most equal to
$2^{n-1}$.
We say that a monomial 
$$X^\bv=x_1^{v_1}x_2^{v_2}\cdots$$ 
in $\C[X]$ is {\em packed} if $\bv$ can be written as $\bc,0,0,0,\dots$
with $\bc$ a composition (\ie a vector whose entries are positive integers).
Thus, the number of packed monomials of degree $n$ is $2^{n-1}$.
Let 
$$P=\sum_{\bv} c_\bv X^\bv$$ 
be a homogeneous polynomial of degree $n$, which is  solution of
\pref{eqfoncQS}
(the sum runs over sequences of non-negative integers of sum $n$).
Associate with each monomial $X^\bv$ an integer 
\begin{equation}
\label{eq:ell}
\ell(X^\bv)=\sum_{i\in\N}i \, v_i.
\end{equation}
Then, all the coefficients of $P$ are determined by those of packed monomials.
To see this, consider a non-packed monomial $X^\bw=x_1^{w_1}x_2^{w_2}\cdots$ 
with $w_i=0$ and $w_{i+1}\neq 0$. 
We substitute $x_i=x_{i+1}=x$ in $P$.
Looking at the monomial
$$x_1^{w_1}\cdots x_{i-1}^{w_{i-1}}x^{w_{i+1}}x_{i+2}^{w_{i+2}}\cdots$$
that does not appear on the right-hand side of \pref{eqfoncQS}, 
we get a linear relation between $c_\bw$ and the coefficients $c_\bv$ of the
monomials $X^\bv$ such that $\ell(X^\bv)<\ell(X^\bw)$,
whence the upper bound on the dimension.

Now, for each composition $I$ of $n$, the stable polynomial $M_I(\X)$
satisfies Equation \eqref{eqfoncQS}. Besides, all $M_I(\X)$ are linearly
independent since, setting $x_{2i+1}=0$ in $\X$ transforms $M_I(\X)$ into
the usual monomial quasi-symmetric functions in even-indexed variables
$M_I(x_2,x_4,x_6,\cdots)$.
\end{proof}

\begin{example}{\rm
Here are the functions $M_I(\X)$ for compositions $I$ of length at most 3:
\begin{align}
    M_{(k)}(\X)    &= - x_1^k + x_2^k - x_3^k + x_4^k - \dots ;\\
M_{(k,\ell)}(\X)   &= \sum_{i} x_{2i+1}^{k+\ell}
                     + \sum_{i<j} (-1)^{i+j} x_i^k x_j^\ell ;\\
M_{(k,\ell,m)}(\X) &= - \sum_{i} x_{2i+1}^{k+\ell+m} + 
\sum_{\gf{i,j}{i<2j+1}} (-1)^i x_i^{k} x_{2j+1}^{\ell+m} \\
\nonumber & \qquad \qquad +
          \sum_{\gf{i,j}{2i+1<j}} (-1)^j x_{2i+1}^{k+\ell} x_j^m
+ \sum_{h<i<j} (-1)^{h+i+j} x_h^k x_i^\ell x_j^m
\end{align}
At least in the first two formulas, it is not hard to check that putting
$x_p=x_{p+1}$ eliminates these variables.}
\end{example}

\section{A new algebra of functions on Young diagrams}\label{SectOnYoung}

\subsection{Interpreting elements of $\Sol$ as functions on Young diagrams}
\label{SubsectActY}

Recall from the introduction that the interlacing coordinates define a map
$\IC$ from Young diagrams to finite sequences of integers.
Besides, a polynomial in infinitely many variables (and in particular the
elements of $\Sol$) can be evaluated on any finite sequence.
Therefore, if $\FF(\Young,\C)$ denotes the algebra of complex-valued 
functions on the set $\Young$ of all Young diagrams, one has a mapping:
\[ \ActY : \begin{array}{rcl} \Sol & \to & \FF(\Young,C) \\
    f &\mapsto & f \circ \IC
\end{array}.\]

\begin{proposition}\label{PropKerActY}
    The kernel of $\ActY$ is the ideal $\langle M_1(\X) \rangle$
    generated by $M_1(\X)=\sum_i (-1)^{i+1} x_i$.
\end{proposition}

\begin{proof}
    It is known that the alternating sum of the interlacing coordinates
    is always zero (see, {\em e.g.}, \cite[Proposition 2.4]{IO02}),
    which means that the function $M_1(\X)$ is in the kernel of $\ActY$.
    \medskip

    For the converse statement, we need the following lemma.
    \begin{lemma}
        Let $f=(f_n)_{n \ge 0}$ be a stable polynomial in $\C[X]$.
        There exists a unique stable polynomial $f'$ such that for any $n$,
        \begin{equation}\label{EqDefFp}
            f_n(x_1,\dots,x_n) = 
        f'_n\left( \sum_{i=1}^n (-1)^i x_i,x_2,\dots,x_n \right).
    \end{equation}
        We denote by $\Phi_{x \to \sum}$ the map  $f\mapsto f'$.
        \label{LemDefFp}
    \end{lemma}
    \begin{proof}[Proof of the Lemma]
        It is clear that, for a fixed integer $n$, the polynomial $f'_n$
        is uniquely determined by Equation \eqref{EqDefFp}:
        \[f'_n(x'_1,\dots,x'_n)=
        f\left( \sum_{i=1}^n (-1)^i x'_i -x'_1,x'_2,\dots,x'_n \right).\]
        It remains to check that the sequence $(f'_n)_{n \ge 0}$
        defines a stable polynomial, which is straightforward.
    \end{proof}
    \noindent {\em End of the proof of the proposition.}
    Let $f$ be a stable polynomial in the kernel of $\ActY$.
    Fix some integer $m$ and a decreasing sequence $x_2>x_3>\dots>x_{2m+1}$
    of {\em negative} integers.
    Set \[x_1 = x_2-x_3+x_4-\dots +x_{2m} - x_{2m+1}.\]
    Then $x_1$ is a positive integer, hence $x_1>x_2$ and, as its
    alternating sum vanishes, $(x_1,\dots,x_{2m+1})$ is 
    the list of interlacing coordinates of some Young diagram.
    Thus,
 \[f(x_1,x_2,x_3,\dots,x_{2m+1})=0.\]
 Let $f'=\Phi_{x \to \sum}(f)$. By Proposition \ref{PropKerActY},
    \[f'(0,x_2,x_3,\dots,x_{2m+1})= f(x_1,x_2,x_3,\dots,x_{2m+1})\]
    and thus, it vanishes.
    This is true for all decreasing lists $(x_2,\dots,x_{2m+1})$ of negative
    integers, so the polynomial $f'(0,x_2,x_3,\dots,x_{2m+1})$ is
    identically zero. In other terms, there exists some polynomial
    $g'_{2m+1}\in \C[x_1,\dots,x_{2m+1}]$ such that
    \[f'_{2m+1} = x'_1 \cdot g'_{2m+1} \text{, or equivalently }
    f_{2m+1} = \left(\sum_{i=1}^n (-1)^i x_i\right) \cdot g_{2m+1}\]
    for some polynomial $g_{2m+1} \in \C[x_1,\dots,x_{2m+1}]$.
    It is straightforward to check that $(g_{2m+1})_{m \ge 0}$ defines a stable
    polynomial and hence 
    \[f=\left(\sum_{i=1}^n (-1)^i x_i\right) \cdot g\]
    for some $g\in \C[X]$, which is what we wanted to prove.
\end{proof}

The image of $\ActY$ is a subalgebra of functions on Young diagrams, that we
shall call from now on \emph{quasi-symmetric functions on Young diagrams} and
denote by $\QLa$.
By definition, it is isomorphic to $\QSym/\langle M_1 \rangle$ and contains
the algebra $\Lambda=\ActY\big( \text{Sym}(\X) \big)$ of 
{\em symmetric functions on Young diagrams}\footnote{See \cref{footnote:symmetric}.}
considered by Kerov and Olshanski.

\begin{corollary}\label{CorolDimQLa}
The homogeneous component of degree $n\ge 2$ of $\QLa$ has dimension $2^{n-2}$.
\end{corollary}

\subsection{Multirectangular coordinates}
\label{SubsectMultirec}

We consider here a different set of coordinates for Young diagrams,
called {\em multirectangular coordinates} and introduced%
\footnote{In fact, R.~Stanley considered coordinates $\pp'$ and $\qq'$ related to
ours by $p'_i=p_i$ and $q'_i=q_1+\dots+q_i$. However, for our purpose,
we prefer the more symmetric version presented here.}
by R.~Stanley in \cite{St1}.

Consider two sequences $\pp$ and $\qq$ of non-negative integers
of the same length $m$.
We associate with these the Young diagram drawn on Figure \ref{FigYDMultirec}.
Note that we allow some $p_i$ or some $q_i$ to be zero, so that the
same diagram can correspond to several sequences (see again Figure \ref{FigYDMultirec}).
Nevertheless, if we require the variables $p_i$ and $q_i$ to be positive,
there is a unique way to associate with a diagram some multirectangular coordinates.
This defines a mapping:
\[ \MRCoord : \{\text{Young diagrams}\}
\to \left\{ \begin{tabular}{c}
    pairs of integer sequences \\
    of the same length
\end{tabular} \right\}.\]

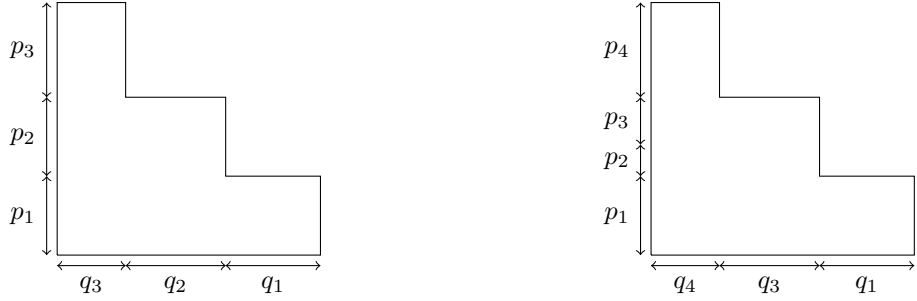
\begin{figure}[t]
    \begin{minipage}[t]{.48 \linewidth}
        \begin{center}
        \begin{tikzpicture}[scale=.7]
            \draw (0,0)--(5,0)--(5,1.5)--(3.2,1.5)--(3.2,3)--(1.3,3)--(1.3,4.8)--(0,4.8)--(0,0);
            \draw[<->] (0,-.2)--(1.3,-.2) node [below,midway] {\footnotesize $q_3$};
            \draw[<->] (1.3,-.2)--(3.2,-.2) node [below,midway] {\footnotesize $q_2$};
            \draw[<->] (3.2,-.2)--(5,-.2) node [below,midway] {\footnotesize $q_1$};
            \draw[<->] (-.2,0)--(-.2,1.5) node [left,midway] {\footnotesize $p_1$};
            \draw[<->] (-.2,1.5)--(-.2,3) node [left,midway] {\footnotesize $p_2$};
            \draw[<->] (-.2,3)--(-.2,4.8) node [left,midway] {\footnotesize $p_3$};
        \end{tikzpicture}
        \end{center}
    \end{minipage}\hfill
    \begin{minipage}[t]{.48 \linewidth}
        \begin{center}
        \begin{tikzpicture}[scale=.7]
            \draw (0,0)--(5,0)--(5,1.5)--(3.2,1.5)--(3.2,3)--(1.3,3)--(1.3,4.8)--(0,4.8)--(0,0);
            \draw[<->] (0,-.2)--(1.3,-.2) node [below,midway] {\footnotesize $q_4$};
            \draw[<->] (1.3,-.2)--(3.2,-.2) node [below,midway] {\footnotesize $q_3$};
            \draw[<->] (3.2,-.2)--(5,-.2) node [below,midway] {\footnotesize $q_1$};
            \draw[<->] (-.2,0)--(-.2,1.5) node [left,midway] {\footnotesize $p_1$};
            \draw[<->] (-.2,1.5)--(-.2,2.1) node [left,midway] {\footnotesize $p_2$};
            \draw[<->] (-.2,2.1)--(-.2,3) node [left,midway] {\footnotesize $p_3$};
            \draw[<->] (-.2,3)--(-.2,4.8) node [left,midway] {\footnotesize $p_4$};
        \end{tikzpicture}
        \end{center}
    \end{minipage}
    \caption{The same Young diagram given by two sets of multirectangular coordinates
    (in the right-hand side, $q_2=0$).}
    \label{FigYDMultirec}
\end{figure}

The multirectangular coordinates are related to interlacing coordinates by the
following changes of variables: for all $i \le m$,
\begin{equation}\label{EqChgtVar}
    \left\{ \begin{array}{l}
        p_i = x_{2i-1}-x_{2i};\\
        q_i = x_{2i}-x_{2i+1};
    \end{array} \right. \hspace{1cm}
     \left\{ \begin{array}{l}    
        x_{2i+1} = (q_{i+1}+\dots+q_m)-(p_1+\dots+p_i); \\
        x_{2i} = (q_i+\dots+q_m)-(p_1+\dots+p_i).  
    \end{array} \right.
\end{equation}
It should not be surprising that there are only $2m$ rectangular coordinates
while there are $2m+1$ interlacing coordinates
because the latter must satisfy the linear relation $\sum_i (-1)^i x_i=0$.

Let us now consider functions on Young diagrams which are polynomials in its
rectangular coordinates.
This amounts  to looking for polynomials in two infinite sets of variables
\[h \left( \begin{array}{ccc}
    p_1 & p_2 & \dots\\
    q_1 & q_2 & \dots
\end{array} \right) \in \C[\bm{p},\bm{q}]\]
satisfying the following two equations:
for all $1 \le i \le m$,
\begin{align}
    \left. h_m\left( \begin{array}{ccc} 
        p_1 & \dots & p_m \\
        q_1 & \dots & q_m
    \end{array} \right)\right|_{q_i=0}
    &=
    h_{m-1}\left( \begin{array}{ccccccc}
        p_1&\dots&p_{i-1}&p_i+p_{i+1}&\dots&p_m\\
        q_1&\dots&q_{i-1}&q_{i+1}&\dots&q_m
    \end{array} \right) ;
    \label{EqQZero} \\
    \left. h_m\left( \begin{array}{ccc} 
        p_1 & \dots & p_m \\
        q_1 & \dots & q_m
    \end{array} \right)\right|_{p_i=0}
    &=
    h_{m-1}\left( \begin{array}{ccccccc}
        p_1&\dots&p_{i-1}&p_{i+1}&\dots&p_m\\
        q_1&\dots&q_{i-1}+q_i&q_{i+1}&\dots&q_m
    \end{array} \right). \label{EqPZero}
\end{align}
In both cases $i=1$ and $i=m$, erase the column containing non-defined
variables.

All these equations express the fact that the images of two sequences
corresponding to the same Young diagram
(as the ones of Figure \ref{FigYDMultirec}) have the same image by $h$.
We denote by $\Sol'$ the space of solutions of this system.

\subsection{Link between the two systems}
\label{SubsectLinkTwoSystems}
In this Section, we study the relation between the two spaces of solutions
$\Sol$ and $\Sol'$.\bigskip

Consider an element $f=(f_n)_{n \ge 0}$ in $\Sol$.
Let $m\ge 0$.
Replace all variables $x_1,\ldots,x_{2m+1}$ in $f_{2m+1}$
according to \eqref{EqChgtVar}, and set
\begin{equation}
h_m\left( \begin{array}{ccc} 
        p_1 & \dots & p_m \\
        q_1 & \dots & q_m
    \end{array} \right) = 
    f_{2m+1}(x_1,\dots,x_{2m+1}).
\end{equation}
Clearly, $h_m$ is a polynomial in $p_1,\dots,p_m,q_1,\dots,q_m$.

Besides, by definition,
\[h_{m+1}\left( \begin{array}{cccc}                                  
        p_1 & \dots & p_m &0 \\                                    
        q_1 & \dots & q_m &0                                      
    \end{array} \right) =
f_{2m+3}(x_1,\dots,x_{2m+1},x_{2m+1},x_{2m+1}).\]
But, as $f$ is an element of $\Sol$, the right-hand side is equal to
$f_{2m+1}(x_1,\dots,x_{2m+1})$.
This implies that $(h_m)_{m \ge 0}$ is an element of $\C[\bm{p},\bm{q}]$.

We will now show that $h$ satisfies Equation \eqref{EqQZero}.
Let us now consider an integer $m \ge 1$ and variables
$p_1,\dots,p_m,q_1,\dots,q_m$.
Assume additionally that $q_i=0$ for some $i$,
which implies $x_{2i}=x_{2i+1}$.
Thus, as $f$ is an element of $\Sol$,
\[f_{2m+1}(x_1,\dots,x_{2m+1})=f_{2m-1}(x_1,\dots,x_{2i-1},x_{2i+2},\dots,x_{2m+1}).\]
Observe that the right-hand side corresponds to the definition of
 \[   h_{m-1}\left( \begin{array}{ccccccc}
        p_1&\dots&p_{i-1}&p_i+p_{i+1}&\dots&p_m\\
        q_1&\dots&q_{i-1}&q_{i+1}&\dots&q_m
    \end{array} \right), \]
which ends the proof of Equation \eqref{EqQZero}.
Equation \eqref{EqPZero} can be proved in a similar way.

\bigskip
Finally, from a stable polynomial $f$ in $\Sol$, we have constructed
an element $h=(h_m)_{m \ge 0}$ in $\Sol'$.
We denote by $\Phixp$ this map (from $\Sol$ to $\Sol'$).

A map $\Phipx$ from $\Sol'$ to $\Sol$ can be constructed 
in a similar way using the first part of Equation \eqref{EqChgtVar}.

\begin{lemma}
    One has $\Phixp \circ \Phipx = \Id_{\Sol'}$.
    Moreover, $\Phixp$ has the same kernel as $\ActY$,
    that is $\langle M_1(\X) \rangle$.
    \label{LemKerChgtCoord}
\end{lemma}

\begin{proof}
    Fix $h \in \Sol'$.
    Let $f= \Phipx(h)$ and $\tilde{h}=\Phixp(f)$.
    By definition,
    \[\tilde{h}_m \left( \begin{array}{ccc} 
        \widetilde{p_1} & \dots & \widetilde{p_m} \\
        \widetilde{q_1} & \dots & \widetilde{q_m}
    \end{array} \right) = f(x_1,\dots,x_{2m+1}), \]
    where $x_1,\dots,x_{2m+1}$ are defined in terms of 
   $\widetilde{p_1},\dots,\widetilde{p_m},\widetilde{q_1},\dots,\widetilde{q_m}$
    using \eqref{EqChgtVar}.
    But 
    \[f(x_1,\dots,x_{2m+1}) = h_m \left( \begin{array}{ccc}
        p_1 & \dots & p_m \\     
        q_1 & \dots & q_m
    \end{array} \right),\]
    where $p_1,\dots,p_m,q_1,\dots,q_m$ are defined in terms of
    $x_1,\dots,x_{2m+1}$ using \eqref{EqChgtVar}.
Applying both relations \eqref{EqChgtVar} {\it in that order}, one
sees directly that for all $i\le m$, one has $p_i= \widetilde{p_i}$
    and $q_i=\widetilde{q_i}$.
    Hence $\tilde{h}_m$ and $h_m$ are the same polynomial,
    which proves the first part of the Lemma.

   For the second one, first observe that $\Phixp(M_1) =0$,
   thus $\langle M_1(\X) \rangle \subset \Ker(\ActY)$.
   Let us prove the opposite inclusion.
   Consider a stable polynomial $f$ such that $\Phixp(f)=0$.
   This implies that $f$ vanishes on all lists of interlacing coordinates of
   Young diagrams, that is $\ActY(f)=0$.
   In other terms, $\Ker(\Phixp) \subset \Ker(\ActY)$.
   The other implication is easy, as we already know by Proposition
   \ref{PropKerActY} that $\Ker(\ActY)=\langle M_1(\X) \rangle$.
\end{proof}

\begin{corollary}\label{CorolSolpEqQLa}
The composition $\ActY'=\ActY \circ \Phipx$ defines an isomorphism
from $\Sol'$ to $\QLa$.
\end{corollary}

\begin{proof}
    Thanks to the previous Lemma, $\Phipx$ is injective
    and by definition of $\QLa$, the map $\ActY$ is surjective.
    It remains to prove that $\Im(\Phipx)$ and $\Ker(\ActY)$
    are complementary subspaces.
   
    As $\Phixp \circ \Phipx = \Id_{\Sol'}$, the composition
    $\Phipx \circ \Phixp$ is a projector.
    Besides, $\Im(\Phipx\circ \Phixp)=\Im(\Phipx)$
    and $\Ker(\Phipx\circ \Phixp)=\Ker(\Phixp)=\Ker(\ActY)$ 
    (the last equality is the second statement of the previous Lemma).
    This ends the proof, as the image and kernel of a projector
    are complementary subspaces.
\end{proof}

The following diagram summarizes the morphisms considered so far:
\begin{equation}\label{EqDiagrammeCommutatif}
    \xymatrix{
\Sol \ar@{->>}[rr]^{\ActY} \ar@/_/@{->>}[rd] & & \QLa \\
& \Sol' \ar@/_/@{^{(}->}[lu] \ar[ru]^{\simeq}_{\ActY'}
  }
  \end{equation}
The two arrows between $\Sol$ and $\Sol'$ correspond respectively to $\Phixp$
and $\Phipx$.

\subsection{The basis $\H_I$}
\label{SubsectH}

In this Section, we exhibit a basis $\H_I$ of $\Sol'$ with the following nice properties
\begin{itemize}
    \item it has an explicit expression
        in terms of multirectangular coordinates ;
    \item if an element of $\Sol$ has an explicit expression interms of multirectangular coordinates,
        we can extract from it its $|H$ expansion.
    \item it is related to some basis of                       
        $\QSym$, whose product  is given by the shuffle of the parts of the 
        compositions indexing the elements, and which can be recognized as
        the one introduced by C.~Malvenuto and C.~Reutenauer in \cite{MR}.
\end{itemize}\medskip

It is here more convenient to work with the original multirectangular
coordinates, as defined by R.~Stanley: set $q'_i=q_i+q_{i+1}+\dots$ for all
$i \ge 1$.
With this change of variables, stable polynomials in the $\bm{q}$ and
$\bm{q'}$ are the same, \eqref{EqQZero} remains the same too, and
\eqref{EqPZero} becomes
\begin{equation}
    \left. h_m\left( \begin{array}{ccc} 
        p_1 & \dots & p_m \\
        q'_1 & \dots & q'_m
    \end{array} \right)\right|_{p_i=0}
    =
    h_{m-1}\left( \begin{array}{ccccccc}
        p_1&\dots&p_{i-1}&p_{i+1}&\dots&p_m\\
        q'_1&\dots&q'_{i-1}&q'_{i+1}&\dots&q'_m
    \end{array} \right) \label{EqPZeroBis}
\end{equation}

Let $I$ be a composition of $n$ with last part greater than $1$.
We define
\begin{equation}\label{EqDefHI}
\H_I\left( \begin{array}{ccc}
      p_1 & p_2 & \dots\\
      q_1 & q_2 & \dots
    \end{array} \right)
=\sum_{s \geq 1}
 \sum_{\gf{I=I_1\cdot I_2\cdots I_s}{k_1 < \dots < k_s}}
     \prod_t \frac{p_{k_t}^{\ell(I_t)}
             (q'_{k_t})^{|I_t|-\ell(I_t)}}{\ell(I_t)!}.
\end{equation}
The summation index $I=I_1\cdot I_2\cdots I_s$ means that we consider
all ways of writing~$I$ as a concatenation of $s$ non-empty compositions.
Here are a few examples (the arguments are omitted for readability):
for $v \ge 2$,
\begin{align}
    \H_{(v)}& =\sum_i p_i (q'_i)^{v-1}; \quad
    \H_{(u,v)}  = \sum_{i<j} p_i (q'_i)^{u-1} p_j (q'_j)^{v-1} + 
    \frac{1}{2} \sum_i p_i^2 (q'_i)^{u+v-2}.
\end{align}

\subsubsection{The $\H_I$ belong to $\Sol'$}

Note that, for any composition $I$ of length $\ell$,
\[\H_I = \sum_{k_1 < \dots < k_\ell} p_{k_1} (q'_{k_1})^{i_1-1}
\dots p_{k_\ell} (q'_{k_\ell})^{i_\ell-1} + \text{non $p$-square free terms.}\]
In particular, the $\H_I$ are linearly independent.

\begin{proposition}
    The functions $(\H_I)$, where $I$ runs over compositions with
last part greater than 1,
    form a basis of $\Sol'$.
\end{proposition}

\begin{proof}
By Corollaries \ref{CorolDimQLa} and \ref{CorolSolpEqQLa}, the dimension
of the homogeneous component of degree $n$ is $2^{n-2}$,
which is exactly the number of functions $\H_I$ of degree $n$
(they are indexed by compositions with last part greater than $1$).
As the functions $(\H_I)_{I \vDash n}$ are linearly independent,
it is enough to prove that they indeed belong to $\Sol'$.\medskip

Fix a composition $I$.
It is straightforward to see that $(\H_I)$ satisfies \eqref{EqPZero}.
Indeed, all monomials that contain $q'_i$ also contain $p_i$.
Let us prove that $(\H_I)$ satisfies \eqref{EqQZero}.

Let $\ell$ be the length of $I$.
Assume that $q_j=0$ for some $j <m$, that is $q'_j=q'_{j+1}$.
Consider Equation \eqref{EqQZero} for $h=\H_I$ and rewrite both sides using
Equation \eqref{EqDefHI}.
The summands corresponding to a sequence of indices $\bm{k}$ that does not
contain $j$ or $j+1$ are the same on both sides.
So let us consider some factorization $f=(I_1,\cdots,I_s)$ of $I$
and some sequence $\bm{k}$ such that $k_t=j$ or $k_t=j+1$ for some $t$.

When $k_t=j$, it may happen that in addition $k_{t+1}=j+1$.
In this case, denote by $\overline{f}$ the factorization 
$(I_1,\dots,I_{t-1},I_t \cdot I_{t+1},\dots,I_s)$, otherwise
set $\overline{f}=f$.
We consider in $\H_I$ the summands corresponding to factorizations $f$
sent to a given factorization
$\overline{f}=(\overline{I_1},\dots,\overline{I_s})$.
We get (recall that we set $q'_{j+1}=q'_j$):
\[M (q'_j)^{|\overline{I_t}|-\ell(\overline{I_t})} \cdot \left(
\frac{p_j^{\ell(\overline{I_t})}}{\ell(\overline{I_t})!} + 
\sum_{\gf{r_1+r_2=\ell(\overline{I_t})}{r_1,r_2 \geq 1}}
\frac{p_j^{r_1}p_{j+1}^{r_2}}{r_1!r_2!}
+ \frac{p_{j+1}^{\ell(\overline{I_t})}}{\ell(\overline{I_{t+1}})!} \right),\]
where $M$ is some monomial in $p_k$ and $q'_k$ ($k\neq j,j+1$).
The first (resp. last) term corresponds to the case where $f=\overline{f}$
and $k_t=j$ (resp. $k_t={j+1}$).
The sum in the middle corresponds to the cases where $k_t=j$ {\em and}
$k_{t+1}=j+1$ and where $\overline{f}$ is obtained from $f$ by gluing two
non-trivial factors of respective lengths $r_1$ and $r_2$.

This expression simplifies {\it via} Newton's binomial formula to
\[ M (q'_j)^{|\overline{I_t}|-\ell(\overline{I_t})} 
\frac{(p_j+p_{j+1})^{\ell(\overline{I_t})}}{\ell(\overline{I_{t+1}})!}, \]
which is the term corresponding to the factorization $\overline{f}$ in the
right-hand side of \eqref{EqQZero}.
 
Finally, we get that $\H_I$ satisfies also \eqref{EqQZero} for $j<m$.
The case $j=m$ must be treated separately, as the column containing
$p_j+p_{j+1}$ in \eqref{EqQZero} does not exist.
However, in this case, it is immediate to check that $\H_I$ satisfies
\eqref{EqQZero} if the last part of $I$ is greater than $1$.
So $\H_I$ lies in $\Sol'$.
\end{proof}

\subsubsection{Some expansions on the basis $\H_I$.}

The fact that the $p$-square free terms of the $\H_I$ are distinct monomials
helps to compute the expansion of a given function on the $\H$-basis.
Here are a few examples.\bigskip

As $\Sol'$ is an algebra (the product of two solutions of \eqref{EqQZero}
and \eqref{EqPZero} is still a solution of these equations),
the product of two $\H_I$ is a linear combination of $\H_I$.
It turns out to have a very simple description.

\begin{proposition}\label{PropProdHI}
For any two compositions $I$ and $J$ (with their last part greater than $1$),
\[\H_I \cdot \H_J = \sum_{K \in I \shuffle J} \H_K,\]
where $I \shuffle J$ denotes the multiset of compositions obtained by
shuffling the {\em parts} of $I$ and $J$ (see Section \ref{SubsectDefQSym}).
\end{proposition}

As this result has already appeared (under a different form) in the literature 
-- see Section \ref{SubSectConfusingAlphabets} --
and as we do not use it in this paper, we only sketch its proof.

\begin{proof}[Sketch of proof]
Both sides of the equality lie in $\Sol'$ and share the same $p$-square free
terms. As the $\H_I$ form a basis of $\Sol'$ and have linearly independent
$p$-square free terms, this is enough to conclude.
\end{proof}

It is also possible to obtain the $\H$-expansions of the functions
$\Phixp \big(M_k(\X)\big)$, which are generators of $\Lambda$.

\begin{proposition}
    For any $k\ge 2$, one has
    \[ \Phixp \big(M_k(\X)\big)
    = \sum_{i=0}^{k-2} (-1)^{i+1} k(k-1)\cdots(k-i+1) \H_{1^i,k-i}.\]
\end{proposition}

\begin{proof}[Sketch of proof]
A lengthy but straightforward computation shows that both sides share the same
$p$-square free terms.
As both sides lie in $\Sol'$, this is enough to conclude.
\end{proof}

Another interesting family of functions on Young diagrams is the one that
originally motivated Kerov's  and Olshanski's work
on symmetric functions on Young diagrams \cite{KO94}: 
fix a partition $\mu$ and define
\[\Ch_\mu(\la) =
\begin{cases}
    |\la| (|\la| -1) \dots (|\la| - |\mu|+1) \, \hat{\chi}^\la_{\mu 1^{|\la|-|\mu|}} 
    & \text{ if } |\la| \ge |\mu| ; \\
    0 & \text{ if } |\la| < |\mu|,
\end{cases}\]
where $\hat{\chi}^\la_{\rho}$ denotes the 
normalized irreducible character values of the symmetric group
({\em normalized} means divided by the dimension).
Kerov and Olshanski have shown that, for any partition $\mu$,
the function $\Ch_\mu$ lies in $\Lambda$.
Hence it lies also in $Q\Lambda \simeq \Sol'$.

A combinatorial interpretation for the coefficients of these functions written
in terms of coordinates $\bm{p}$ and $\bm{q'}$ is given in
\cite{FerayPreuveStanley}.
Using the material here, we can directly deduce from it a combinatorial
interpretation for the coefficients of their $\H$-expansion.

Fix an integer $k \ge 1$, two permutations $\si$ and $\tau$ in
the symmetric group $S_k$ and $\varphi$ a bijection from the set
$C(\si)$ of cycles of $\si$ to $\{1,\dots,|C(\si)|\}$.
Then, for a cycle $c'$ of $\tau$ we define
\[\psi(c') = \max_{\gf{c \in C(\si)} {c \cap c' \neq \emptyset}} \varphi(c)\]
and the composition $I_{(\si,\tau)}^\varphi$ of length $|C(\si)|$
whose $j$-th part is $1+|\psi^{-1}(j)|$.
\begin{proposition}
    Fix a partition $\mu$ of length $k$ and 
    choose arbitrarily a permutation $\pi$ in $S_k$ of cycle-type $\mu$.
    Then
    \[\Ch_\mu = \sum_{\gf{(\sigma,\tau) \in S_k}{\sigma \, \tau =\pi}}
    \varepsilon(\tau) \sum_{\gf{\varphi \text{ bijection}}{C(\si) \to \{1,\dots,|C(\si)|\}}}
    \H_{I_{(\si,\tau)}^\varphi}.\]
\end{proposition}
\begin{proof}
    From the main result of \cite{FerayPreuveStanley} and the definition of $\H_I$,
    it is clear that the $p$-square free terms of both sides coincide.
    As both sides lie in $\Sol'$, the equality holds.
\end{proof}
\begin{example}{\rm
    Consider $\mu=(3)$ (thus $k=3$) and choose $\pi=(1\ 2\ 3)$.
    Then $\pi$ has 6 factorizations in two factors in $S_3$:
    $\pi=(1\ 2\ 3) \, \Id_3$ yields a term $\H_{(4)}$,
    the three factorizations as product of two transpositions yield each
    $-\H_{(1,3)}-\H_{(2,2)}$, the factorization $\pi=\Id_3 \, (1\ 2\ 3)$
    yields a term $6\H_{(1,1,2)}$ and finally $\pi=(1\ 3\ 2)^2$ yields $\H_{(2)}$.
    So we get the $\H$-expansion of $\Ch_{(3)}$:
    \[ \Ch_{(3)} = \H_{(4)}-3\H_{(1,3)}-3\H_{(2,2)}+6\H_{(1,1,2)}+\H_{(2)}.\]
    }
\end{example}
\label{SubsubSectHExp}
\subsubsection{Collapsing the alphabets.}
\label{SubSectConfusingAlphabets}

Consider now the morphism which sends $p_j$ and $q'_j$ to the same variable
$y_j$. The image of $\H_I$ is the following quasi-symmetric function of the
alphabet $Y$
\begin{equation}\label{HIX}
    \sum_{s \geq 1} \sum_{I=I_1\cdot I_2\cdots I_s}
    \prod_t \frac{1}{\ell(I_t)!}\, M_{|I_1|,|I_2|,\dots,|I_s|}(Y).
\end{equation}
Call $H_I(Y)$ this function.
Then, $H_I$, when $I$ runs over all compositions,
form a basis of $\QSym$ since the transition matrix
with the monomial basis is triangular.
Proposition \ref{PropProdHI} implies that their product is given by 
the shuffle operation on the parts of the compositions.

These functions already appear in \cite[Equation (2.12)]{MR},
where they are defined as the dual basis of some noncommutative symmetric
functions denoted by $\Phi^I$ in \cite{NCSF1} (analogues of power sums).
The shuffle  property is clear in this context, as the $\Phi^I$
form a multiplicative basis on primitive generators.

\subsection{The functions $N_G$}\label{SubsectNG}

In this Section, we study a family of functions on Young diagrams
indexed by bipartite graphs with two types of edges.

Let $G$ be an unlabelled bipartite graph with vertex set
$V=V_1 \sqcup V_2$ and edge set $E=E_{1,2} \sqcup E_{2,1} \subset V_1 \times V_2$.
We consider the polynomial $N_G$ in the variables $p_i$ and $q_i$
 defined as follows:
\[ N_G\left( \begin{array}{ccc}
    p_1 & p_2 & \dots\\
    q_1 & q_2 & \dots
\end{array} \right) = \sum_{r} \left( \prod_{v_1 \in V_1} p_{r(v_1)}
\prod_{v_2 \in V_2} q_{r(v_2)} \right),\]
where the sum runs over functions $r : V \to \N$ satisfying
the following {\em order condition} (the large and strict inequalities are important!):
\begin{itemize}
    \item for each edge $e=(v_1,v_2)$ in $E_{1,2}$, one has $r(v_1) \leq r(v_2)$ ;
    \item for each edge $e=(v_1,v_2)$ in $E_{2,1}$, one has $r(v_2) < r(v_1)$.
\end{itemize}

\begin{example}{\rm
    Consider the graph $G_\ex$ drawn on Figure \ref{fig:example_unlabelled_graph}.
    Vertices in $V_1$ (resp. $V_2$) are drawn in white (resp. black).
    edges in $E_{1,2}$ (resp. $E_{2,1}$) are represented directed from their extremity in 
    $V_1$ to their extremity in $V_2$ (resp. from their extremity in $V_2$
    to their extremity in $V_1$).

    \begin{figure}
        \[\begin{tikzpicture}
            \tikzstyle{bv}=[circle,fill=black,inner sep=0pt,minimum size=2mm]
            \tikzstyle{wv}=[circle,draw=black,inner sep=0pt,minimum size=2mm]
            \node[wv] (v1) at (0,0) { };
            \node[wv] (v2) at (0,-1) { };
            \node[bv] (v3) at (1,0) { };
            \node[bv] (v4) at (1,-1) { };
            \node[wv] (v5) at (2,-.5) { };
            \node[bv] (v6) at (3,-.5) { };
            \draw[->] (v1) -- (v3);
            \draw[->] (v1) -- (v4);
            \draw[->] (v2) -- (v3);
            \draw[->] (v2) -- (v4);
            \draw[->] (v3) -- (v5);
            \draw[->] (v4) -- (v5);
            \draw[->] (v5) -- (v6);
        \end{tikzpicture}\]
        \caption{Example of an unlabelled bipartite graph $G$.}
        \label{fig:example_unlabelled_graph}
    \end{figure}
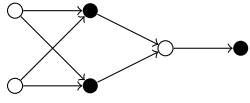

    Let $r$ be a function from its vertex set to $\N$.
    Denote by $e$ and $f$ the images of the leftmost white vertices, 
    by $g$ and $h$ the images of the black vertices just to their right,
    then by $i$ the image of the white vertex to their right and finally by $j$ 
    the image of the rightmost black vertex.

    Then, by definition, $r$ satisfies the order condition if and only if
    \[e,f \le g,h < i \le j.\]
Note the alternating large and strict inequalities.
Finally, one has
\[N_{G_\ex}\left( \begin{array}{ccc}
        p_1 & p_2 & \dots\\       
            q_1 & q_2 & \dots         
        \end{array} \right) = \sum_{e,f \le g,h < i \le j}
        p_e p_f p_i q_g q_h q_j.\]
    }
\end{example}

\begin{remark}[Why are we considering this family ?]{\rm
In the case where we have only edges of the first type, that is $E_{2,1} = \emptyset$,
the functions have played an important role in some recent works on
irreducible character values of symmetric groups;
indeed, the function $\Ch_\mu$ defined in \cref{SubsubSectHExp}
writes in a combinatorial way as a sum of $N_G$ functions,
see {\em e.g.} \cite[Theorem 2]{FS11}.
This expansion has proved useful for studying the asymptotics of $\Ch_\mu$
\cite{FS11} and to answer a question of Kerov on the expansion of $\Ch_\mu$
in the so-called {\em free cumulant} basis of $\Lambda$ 
see~\cite{Fer09,DFS10}.

The extension to the case $E_{2,1} \neq \emptyset$ will be useful in 
Section~\ref{subsec:explicit_family}.
}
\end{remark}

\begin{lemma}
\label{LemNGInSol}
Let $G$ be a bipartite graph as above.
Assume that each element in $V$ is the extremity
of at least one edge in $E_{1,2}$.
Then the polynomial $N_G$ belongs to $\Sol'$.
\end{lemma}

\begin{proof}
    Let us check that $N_G$ satisfies equation~\eqref{EqQZero}. 
    We define 
    \[ \begin{cases}
        p'_j =p_j & \text{ if }j<i ;\\
        p'_i =p_i+p_{i+1} ;& \\
        p'_j =p_{j+1} & \text{ if }j>i ;
    \end{cases} \qquad
    \begin{cases}  
        q'_j =q_j & \text{ if }j<i ;\\
        q'_j =q_{j+1} & \text{ if }j \ge i.
    \end{cases}\]
    These are the variables in the right-hand side of~\eqref{EqQZero}.
    Consider first the left-hand side:
    \[N_G \left.\left( \begin{array}{ccc}                               
            p_1 & \dots & p_m\\                                       
                q_1 & \dots &q_m                                       
            \end{array} \right)\right|_{q_i=0} = \sum_{r} \left( \prod_{v_1 \in V_1} p_{r(v_1)} 
            \prod_{v_2 \in V_2} q_{r(v_2)} \right),\]
where the sum runs over functions $r : V \to \{1,\dots,m\}$ satisfying the order condition.
Besides, one can restrict the sum to functions $r$ such that
$r(v_2) \neq i$ for any $v_2 \in V_2$
(we call them $i$-avoiding functions).

If $i=m$, an $i$-avoiding function $r$ which satisfies the order condition
cannot associate $m$ with a vertex in $V_1$
(indeed, as each vertex $v_1$ in $V_1$ is the extremity of at least one edge
$(v_1,v_2)$ in $E_{1,2}$, so that  $r(v_1) \le r(v_2) <m$).
Therefore equation~\eqref{EqQZero} is satisfied for $i=m$.

Let us consider now the case $i<m$.
With any function $r$, we associate a function $r'=\Phi(r) : V \to \{1,\dots,m-1\}$ defined as follows:
\[r'(v) = r(v) \text{ if }r(v) \le i \text{ and } r'(v) = r(v)-1 
\text{ if }r(v) > i.\]
It is straightforward to check that, if $r$ is $i$-avoiding, $r'$ satisfies the order condition.
Indeed, the only problem which could occur is that $r(v_2)=i$ and $r(v_1)=i+1$
for an edge $(v_1,v_2)$ in $E_{2,1}$, but we forbid $r(v_2)=i$.

The preimage  of a given function $r'$ is obvious:
it is the set of functions $r$ with
\[\begin{cases}
    r(v) = r'(v) & \text{ if }r'(v) < i ;\\
    r(v) \in \{i;i+1\} & \text{ if }r'(v) = i ;\\
    r(v)= r'(v)+1 & \text{ if }r'(v) > i.
\end{cases}\]
If $r'$ satisfies the order condition, all its $i$-avoiding pre-images $r$ 
also satisfy the order condition.
Once again, the only obstruction to this would be in the case when $r'(v_1)=r'(v_2)=i$
for some edge $(v_1,v_2)$ in $E_{1,2}$.
In this case, preimages $r$ with $r(v_1)=i+1$ and $r(v_2)=i$ would not
satisfy the order condition but we forbid $r(v_2)=i$.

Now, using the above description of the preimage,
for any function $r'$, one has:
\[\sum_{r \in \Phi^{-1}(r')} \left( \prod_{v_1 \in V_1} p_{r(v_1)}      
\prod_{v_2 \in V_2} q_{r(v_2)} \right)
= \left( \prod_{v_1 \in V_1} p'_{r'(v_1)}      
\prod_{v_2 \in V_2} q'_{r'(v_2)} \right),\]
Summing over all functions $r': V \to \{1,\dots,m-1\}$ with the order condition,
we get equality~\eqref{EqQZero}.

The proof of \eqref{EqPZero} is similar.
\end{proof}

\begin{remark}{\rm
In~\cite[Section 1.5]{FS11}, an equivalent definition of $N_G$ as
a function on Young diagrams is given in the case $E_{2,1} = \emptyset$.
The sole fact that $N_G$ can be defined using only the Young diagram and
not its rectangular coordinates explains that it belongs to $\Sol'$.
}
\end{remark}

As $N_G$ belongs to $\Sol'$, its image $\Phipx(N_G)$ is an element of $\Sol$,
that is, some quasi-symmetric function evaluated on the virtual alphabet $\X$.
We shall now determine this quasi-symmetric function.

Define $F_G$ as the generating series of functions on $G$ satisfying the order condition:                          
\[F_G(u_1,u_2,\dots)=\sum_{\gf{r:V \to \N}{\text{with order condition}}}
\prod_{v \in V} u_{r(v)}.\]
Note that $F_G$ is a quasi-symmetric function.

\begin{proposition}
\label{PropFormulaNG}
For any bipartite graph $G$ with vertex set $V=V_1 \sqcup V_2$,
\[ N_G\left( \begin{array}{ccc}
p_1 & p_2 & \dots\\
q_1 & q_2 & \dots
\end{array} \right) = (-1)^{|V_1|} \Phixp\big(F_G(\X)\big).\]
\end{proposition}

\begin{proof}
The stable polynomial $\Phipx(N_G)$ is an element of $\Sol$, that is $F(\X)$
for some quasi-symmetric function $F$.
To identify $F$, we shall send all odd-indexed variables $x_{2i+1}$ to $0$.
This amounts to sending $p_i$ to $-x_{2i}$ and $q_i$ to $x_{2i}$.
Thus we get that
\[F(x_2,x_4,\cdots) = (-1)^{|V_1|} F_G(x_2,x_4,\cdots).\]
This implies $F=(-1)^{|V_1|} F_G$ and the proposition follows
by applying $\Phixp$.
\end{proof}

\begin{remark}{\rm
We have just shown that the series $N_G$ in two sets of variables
can be recovered from the corresponding series $F_G$ in one set of variables.
This is a bit surprising and
it is quite remarkable that Young diagrams give a natural way of doing this.
}
\end{remark}

\begin{remark}{\rm 
In the above proof, to identify $F$, we could also have sent all even-indexed
variables $x_{2i}$ to zero.
Under this specialization,
\[F(\X) = F(\ominus(x_1) \ominus (x_3) \ominus \dots) = F(\ominus (\dots,x_3,x_1)) =S(F) (\dots,x_3,x_1),\]
{\em i.e.} the antipode of $F$ evaluated in the {\em reverted} alphabet       
$(\dots, x_3,x_1)$.

On the other side $\Phipx(N_G)$, when $x_{2i}=0$, consists in
plugging $p_{i+1}=-q_i=x_{2i+1}$ in $N_G$.
Thus we get, up to a sign $(-1)^{|V_2|}$,
the generating series of functions $r$ on the graph $G$ verifying some modified order condition:
just take the order condition above and exchange large and strict inequalities
(this modification of type of equalities comes from the above index shift).

Finally, as $F=(-1)^{|V_1|} F_G$,
we get that $S(F_G)$ is, up to a sign $(-1)^{|V|}$, the generating
series $F_{G'}$ where $G'$ is obtained from $G$ by exchanging the sets
$E_{1,2}$ and $E_{2,1}$.

This result can be extended to general directed graphs with two types of edges
(here, we only considered the "bipartite" case)
and is a direct consequence of the fundamental lemma on $P$-partitions
\cite[Theorem 6.2]{StOrderedStructure} and of the image of
a fundamental quasi-symmetric function by the antipode
\cite[Corollary 2.3]{MR}.
}\end{remark}

\section{Noncommutative generalization}
\label{SectNonCommutative}

Let now $A$ be a noncommutative ordered alphabet. We shall find all
noncommutative polynomials $P(A)$ satisfying the noncommutative version of
Equation~\eqref{eqfoncQS}, that is, Equation~\eqref{eqfoncWQS}, reproduced here
for convenience:
\begin{equation*}
P(a_1,a_2,\dots)|_{a_i=a_{i+1}} =
P(a_1,\dots,a_{i-1},a_{i+2},\dots).
\end{equation*}

As in the commutative framework, homogeneous polynomials in infinitely many variables are formally
sequences of homogeneous polynomials $(R_n)_{n \ge 1}$,
where $R_n$ lies in the algebra $\C\langle a_1,..., a_n \rangle$ (polynomials
in $n$ noncommuting variables)  such that
$R_{n+1}(a_1,\ldots,a_n,0)=R_n(a_1,\ldots,a_n)$.
The vector space of these stable noncommutative polynomials will be denoted $\C\langle A\rangle$.
\subsection{Word quasi symmetric functions}
\label{SubsectDefWQSym}

The natural noncommutative analogue of $\QSym$ is the algebra of
\emph{word quasi symmetric functions}, denoted by $\WQSym$.
We recall here its construction (see, {\it e.g.}, \cite{NT-tri}).

Monomials in the noncommutative framework are canonically indexed
by finite words $w$ on the alphabet $\N$ as follows
\[\bm{a}_w = a_{w_1} \, a_{w_2} \dots a_{w_{|w|}}. \]
The {\em evaluation} $\eval(w)$ of a word $w$ is the integer sequence $v=(v_1,v_2,\dots)$,
where $v_i$ is the number of letters $i$ in $w$.
Then the commutative image of $\bm{a}_w$ is $\bm{X}^{\eval(w)}$.

In the noncommutative framework, set compositions play the role of
compositions.
A {\em set composition} of $n$ is an {\em (ordered)} list $(I_1,\dots,I_p)$
of pairwise disjoint non-empty subsets of $\{1,\dots,n\}$,
whose union is $\{1,\dots,n\}$.

Such an object can be encoded by a word $w$ of length $n$ defined as follows:
$w_i=j$ if $i \in I_j$.
This encoding is injective.
Words obtained that way are exactly those satisfying the following property:
for $j\ge 1$, if the letter $j+1$ appears in $w$, then $j$ appears in $w$.
Such words are called {\em packed}.
Equivalently, a word $w$ is packed if and only if its evaluation $\eval(w)$
can be written as $\bm{c},0,0,\dots$, where $\bm{c}$ is a composition (that is,
a vector of positive integers).

With each word $w$ is associated a packed word, called {\em packing} of $w$
and denoted $\pack(w)$:
replace all occurrences of the smallest letter appearing in $w$ by $1$
then, occurrences of the second smallest letter by $2$ and so on.
For example the packing of $3\, 6\, 4\, 4\, 3\, 4$
is $1\, 3\, 2\, 2\, 1\, 2$.

Depending on the point of view, it may be more convenient to use set
compositions or packed words. In this section, we use packed words while the
next one deals with set compositions.

By definition, $\WQSym$ is a subalgebra of the algebra of stable polynomials
in noncommuting variables $a_1,a_2,\dots$.
A basis of $\WQSym$ is given as follows:
\[ P_u = \sum_{\gf{w \text{ s.t.}}{\pack(w)=u}} \bm{a}_w \]
Note that the commutative image of $P_u$ is $M_{\eval(u)}$.

Conversely, if $u$ is a non-decreasing packed word,
$P_u$ is the unique noncommutative stable polynomial whose commutative
image is $M_{\eval(u)}$ and in which all monomials have letters in nondecreasing order.
Now, for any packed word $u$, consider its nondecreasing rearrangement $u^\uparrow$
and the smallest permutation $\sigma_u$ in $\SG_{|u|}$ sending $u^\uparrow$ to $u$.
Then, the polynomial $P_u$ is obtained by letting $\sigma_u$
act on all monomials of $P_{u^\uparrow}$
(note that all monomials in $P_{u^\uparrow}$ have $|u|$ letters).
These two properties characterize the elements $P_u$.

To finish, let us mention that 
the ordered Bell numbers $OB(n)$~\cite[A000670]{Sloane} count packed words,
thus give the dimension of the homogeneous subspace of degree $n$ of $\WQSym$.

\subsection{$\WQSym$ and a virtual alphabet}
\label{SubsectEvalNCAlph}

We want to make sense of the virtual alphabet
\begin{equation*}
\A = \ominus (a_1) \oplus (a_2) \ominus (a_3) \oplus (a_4) \ominus \dots
\end{equation*}
We cannot define $\ominus$ by means of the antipode of
$\WQSym$, since it is not involutive.

Rather, we define $\WQSym(\A)$ as follows: if $u$ is a nondecreasing packed
word, $P_u(\A)$ is the noncommutative analogue of $M_{\eval(u)}(\X)$
where $x_k$ is replaced by $a_k$ and all letters
in any monomial of $P_u(\A)$ are in nondecreasing order.
Now, for any $u$, the (stable) polynomial $P_u(\A)$ is obtained by letting $\sigma_u$
act on all monomials of $P_{u^\uparrow}(\A)$,
where $\sigma_u$ and $u^\uparrow$ are defined as above.

Finally, for $F$ in $\WQSym$, we define $F(\A)$ by linearity.
For example,
\begin{align}
P_{1^k}(\A) &= - a_1^k + a_2^k - a_3^k + a_4^k - \dots;\\
P_{1^k 2^\ell}(\A) &= \sum_{i} a_{2i+1}^{k+\ell} + 
    \sum_{i<j} (-1)^{i+j} a_i^k a_j^\ell;\\
P_{112}(\A) &= \sum_{i} a_{2i+1}^3 + 
    \sum_{i<j} (-1)^{i+j} a_i a_i a_j;\\
P_{121}(\A) &= \sum_{i} a_{2i+1}^3 + 
    \sum_{i<j} (-1)^{i+j} a_i a_j a_i;\\
P_{211}(\A) &= \sum_{i} a_{2i+1}^3 + 
    \sum_{i<j} (-1)^{i+j} a_j a_i a_i;
\end{align}

It is not hard to check that plugging $a_p=a_{p+1}$ for some $p$
in the noncommutative polynomials above eliminates these variables.
This property will be established for any packed word $u$ in the next section.

Finally, we denote by $\WQSym(\A)$ the following subspace of $\C\langle A \rangle$:
\[\WQSym(\A) = \{ P(\A), P \in \WQSym \}\]
We will see in next section that $\WQSym(\A)$ is an algebra and 
that $F \mapsto F(\A)$ is an isomorphism of algebras from 
$\WQSym$ to $\WQSym(\A)$.

\subsection{Solution of Equation \eqref{eqfoncWQS}}

\begin{theorem}
A polynomial $P$ satisfies Equation (\ref{eqfoncWQS})
if and only if $P$ belongs to $\WQSym(\A)$.
\end{theorem}

\begin{proof}
We introduce the following ring homomorphism 
$$\begin{array}{rcl}
\phi:\C\langle A\rangle&\rightarrow&\C[X] ;\\
a_i&\mapsto& x_i.
\end{array}$$

We first prove that  the dimension of the space of homogeneous polynomials in
$\C[A]$ of degree $n$ which satisfy \pref{eqfoncWQS} is at most equal to
$OB(n)$ (ordered Bell number).
We may observe that a word $w\in A^{*}$ is packed iff $\phi(w)$ is packed as a
monomial in $\C[X]$. Let 
$$P=\sum_{w\in A^{*}} c_w w$$ 
be a homogeneous polynomial of degree $n$, solution of \pref{eqfoncWQS}.
The goal is to prove that all the coefficients of $P$ are determined by 
its coefficients on packed words.
Let $w$ be a non-packed word and $(a_i,a_{i+1})$ be a  pair of letters such
that: $a_i$ does not appear in $w$, and $a_{i+1}$ does. 
We let $a_i=a_{i+1}=a$ in $P$.
By looking at the coefficient of the word obtained by replacing $a_{i+1}$
by $a$ in $w$, which does not appear on the right-hand side of
\pref{eqfoncWQS}, we get a linear relation between $c_w$ and coefficients
$c_v$ of words $v$ such that $\ell(\phi(v))<\ell(\phi(w))$ ($\ell$ is defined
in \eqref{eq:ell}), whence the upper bound on the dimension.

Let us now prove that, if $P\in\WQSym(\A)$, it satisfies the functional
Equation~(\ref{eqfoncWQS}).
First consider the case $P=P_u(\A)$, with $u$ nondecreasing.
By definition, $P_u$ contains only monomials with variables with
{\em non-decreasing indices}.
This is still true after substitution $a_{i+1}=a_i$ as we 
equate two {\em consecutive} variables.
Therefore any cancellation occurring in the commutative image,
is mimicked in the noncommutative version.
As $M_{\text{eval}(u)}(\X)$, which is the commutative image of $P_u(\A)$,
is a solution of Equation \eqref{eqfoncQS}, the noncommutative
polynomial $P_u(A)$ is a solution of  \eqref{eqfoncWQS}.

Moreover, if it is true for a polynomial $P$, then it is true for any
polynomial obtained by action of a permutation on it, so any element of
$\WQSym(\A)$ satisfies~\eqref{eqfoncWQS}.
Besides, 
the $P_u(\A)$ are linearly independent since, setting $a_{2i+1}=0$, one
transforms $P_u(\A)$ into the usual word quasi-symmetric function in
even-indexed variables $P_u(a_2,a_4,\dots)$.
This gives a lower bound on the dimension of the space of solutions
of Equation \eqref{eqfoncWQS},
corresponding to the upper bound found above, which ends the proof.
\end{proof}

We can now prove the following.
\begin{corollary}\label{corol:FtoFA_Morphism}
The space $\WQSym(\A)$ is a subalgebra of $\C\langle A\rangle$ and 
$F \mapsto F(\A)$ is an isomorphism of algebras from 
$\WQSym$ to $\WQSym(\A)$.
\end{corollary}
\begin{proof}
The space of solution of Equation \eqref{eqfoncWQS} is clearly a subalgebra of $\C\langle A\rangle$.
But we have just proved that it is $\WQSym(\A)$, whence our first claim.

Surjectivity of $F \mapsto F(\A)$ comes directly from the definition and
injectivity corresponds to the fact the $P_u(\A)$ are linearly independent, proved above.
Thus the only thing to prove is that it indeed defines an algebra morphism.

Let $F$ and $G$ be two elements of $\WQSym$ and,
to avoid confusion, denote by $F\star G$ their product in $\WQSym$.
As $\WQSym(\A)$ is an algebra, the stable polynomial $F(\A) \cdot G(\A)$
(here, the product is the product in $\C\langle A\rangle$) can be written as $H(\A)$,
for some element $H \in \WQSym$.
But setting $a_{2i+1}=0$ sends $F(\A)$, $G(\A)$ and $H(\A)$ to the {\em usual}
word quasi symmetric functions $F$, $G$ and $H$ in even indexed variables.
Therefore this specialization also sends $F(\A) \cdot G(\A)$ to $F\star G$ in even indexed variables.
Thus, necessarily, $H=F\star G$, which concludes the proof.
\end{proof}

\begin{remark}{\rm
This morphism property is natural to look for when defining evaluation on virtual alphabets.
It is quite remarkable that our functional equation helps to prove it. 
}\end{remark}

\section{Non commutative multirectangular coordinates}
\label{SectNonCommutativeMultirect}

\subsection{What can be extended to the noncommutative framework?}
We would like to lift the sets and maps of the commutative diagram
\eqref{EqDiagrammeCommutatif} to the noncommutative framework.

In the previous Section, we have studied the structure of 
$\SolNC$, the noncommutative analog of $\Sol$.
The algebra $\QLa$, defined as a subalgebra of functions on Young diagrams,
has no noncommutative analog.
A noncommutative analog of $\Sol'$ is easily defined as follows.

Consider the space $\SolNC'$ of polynomials in two infinite sets of
non commutative variables satisfying
\begin{align}
    \left. h_m\left( \begin{array}{ccc} 
        b_1 & \dots & b_m \\
        d_1 & \dots & d_m
    \end{array} \right)\right|_{d_i=0}
    &=
    h_{m-1}\left( \begin{array}{ccccccc}
        b_1&\dots&b_{i-1}&b_i+b_{i+1}&\dots&b_m\\
        d_1&\dots&d_{i-1}&d_{i+1}&\dots&d_m
    \end{array} \right) 
    \label{EqDZero} \\
    \left. h_m\left( \begin{array}{ccc} 
        b_1 & \dots & b_m \\
        d_1 & \dots & d_m
    \end{array} \right)\right|_{b_i=0}
    &=
    h_{m-1}\left( \begin{array}{ccccccc}
        b_1&\dots&b_{i-1}&b_{i+1}&\dots&b_m\\
        d_1&\dots&d_{i-1}+d_i&d_{i+1}&\dots&d_m
    \end{array} \right) \label{EqBZero}
\end{align}

We also consider the substitutions
\begin{equation}\label{EqChgtVarNC}
    \left\{ \begin{array}{l}
        b_i = a_{2i-1}-a_{2i}\\
        d_i = a_{2i}-a_{2i+1}
    \end{array} \right. \hspace{1cm}
     \left\{ \begin{array}{l}    
        a_{2i+1} = (d_{i+1}+\dots+d_m)-(b_1+\dots+b_i) \\
        a_{2i} = (d_i+\dots+d_m)-(b_1+\dots+b_i)     
    \end{array} \right.
\end{equation}

Then we have the following:
\begin{proposition}
    The substitutions~\eqref{EqChgtVarNC} define morphisms
\[\begin{array}{rccl}
    \Phiab:& \SolNC &\to &\SolNC' \\
    \Phiba:& \SolNC' &\to &\SolNC,
\end{array}\]
with $\Phiab \circ \Phiba = \Id_{\SolNC'}$.
In particular $\Phiab$ is surjective.
\end{proposition}
\begin{proof}
    The arguments given in the commutative setting in Section \ref{SubsectLinkTwoSystems}
    work without any changes in the noncommutative setting.
\end{proof}

A natural question is to determine the kernel of $\Phiab$.
This kernel is a two-sided element of $\WQSym(\A)$
and thus, by Corollary \ref{corol:FtoFA_Morphism},
can be identified to a two-sided element of $\WQSym$.
Clearly $P_1(\A)=\sum (-1)^i a_i$ lies in this kernel,
and hence, the two-sided ideal generated by $P_1$ is included in it.

Recall that the symmetric group $\SG_n$ acts on polynomials of degree $n$ in
non-commuting variables by permuting the letters inside words.
This action stabilizes the homogeneous component $\WQSym_n$ of degree $n$ of
$\WQSym$.
Note that $\Phiab$ is compatible with the reordering of variables
inside a noncommutative monomial and, hence,
the homogeneous part of degree $n$
of its kernel is invariant by the action of $\SG_n$.

We shall prove in this Section that the kernel of $\Phiab$
is indeed the smallest two-sided ideal containing $P_1$ and 
stable by reordering variables.\bigskip

{\bf Notation.}
Let us fix some notations and conventions for the symmetric
group and its action.

Applying the product $\sigma \tau$ of two permutations 
$\sigma$ and $\tau$ means applying $\sigma$ and {\em then}
$\tau$.
Then the symmetric group $S_n$ has natural right actions
on positive integers smaller or equal to $n$,
subsets of $\{1,\dots,n\}$, words of length $n$,
noncommutative polynomials of degree $n$.
We shall denote the image of an object $O$ by a permutation $\sigma$
by $O \cdot \sigma$.
By example
\begin{align*}
    \{1,2\} \cdot 3124 &= \{1,3\}\\
 ( a_{i_1} a_{i_2} a_{i_3} a_{i_4} ) \cdot 3124
 &= a_{i_3} a_{i_1} a_{i_2} a_{i_4}.
 \end{align*}

\subsection{A lower bound on the dimension of the kernel}
\label{SubsectDimConjKernel}

Let us denote by $\KK_n$ the subspace of $\WQSym_n$ obtained by taking:
\begin{itemize}
    \item elements of degree $n$ of the {\em left}
        ideal generated by $P_1$;
    \item their images by the action of permutations in $\SG_n$.
\end{itemize}

Clearly, $\KK_n$ is included in the kernel of $\Phiab$
(it is a priori smaller than the degree $n$ component
of the smallest two-sided ideal containing $1$ and 
stable by the action of $\SG_n$).
We shall determine the dimension of $\KK_n$ and, hence,
a lower bound on the dimension of the kernel of $\Phiab$
in degree $n$.

Let us first introduce a few more notations.
For each $n$, choose a complementary subspace $\U_n$ of $\KK_n$ in
$\WQSym_n$, stable by the action of $\SG_n$
(as $\KK_n$ is stable by the action of the {\em finite} group $\SG_n$,
the existence of a stable complementary subspace is a classical lemma
in representation theory, see, {\em e.g.},
\cite[Proposition 1.5]{FultonHarris}).
Then construct subspaces $M^{n}_E \subset \WQSym_n$ indexed 
by subsets $E$ of $\{1,\dots,n\}$ as follows.
\begin{itemize}
\item If $E=\{n-k+1,\dots,n\}$ for some $k \in \{0,\dots,n \}$,
      then set $M^{n}_E=U_{n-k} \cdot P_1^{k}$.
\item Otherwise, set $k=|E|$ and
      choose a permutation $\sigma_E\in\SG_n$ such that one has 
      $\{n-k+1,\dots,n\} \cdot \sigma_E=E$ and define 
      $M^{n}_E=M^{n}_{\{n-k+1,\dots,n\}} \cdot \sigma_E $.
\end{itemize}

\begin{lemma}
\label{LemMEWellDefined}
The space $M^{n}_E$ is well-defined.
Moreover, if $\sigma \in \SG_n$ and $E$ is a subset of $\{1,\dots,n\}$, then
$M^n_E \cdot \sigma=M^n_{E \cdot \sigma}$.
\end{lemma}

\begin{proof}
We need to check that $M^{n}_{\{n-k+1,\dots,n\}} \cdot \sigma_E$ 
does not depend of the chosen permutation $\sigma_E$.
Let $\sigma_E$ and $\sigma'_E$ be permutations of size $n$ with
$$\{n-k+1,\dots,n\} \cdot \sigma_E=\{n-k+1,\dots,n\} \cdot \sigma'_E=E.$$
Then $\sigma_E=\tau \sigma'_E $ with 
$\{n-k+1,\dots,n\} \cdot \tau=\{n-k+1,\dots,n\}$,
that is $\tau \in \SG_{n-k} \times S_k$.
As $\{U_{n-k}\}$ and $P_1^k$ are respectively stable by the actions
of $\SG_{n-k}$ and $\SG_k$,
the space $M^{n}_{\{n-k+1,\dots,n\}}=U_{n-k} P_1^k$
is stable by $\tau$.
Hence
\[M^{n}_{\{n-k+1,\dots,n\}} \cdot \sigma_E= 
M^{n}_{\{n-k+1,\dots,n\}} \cdot \tau \sigma'_E
=M^{n}_{\{n-k+1,\dots,n\}} \cdot \sigma'_E,\]
and $M^{n}_E$ is well-defined.

To prove the second claim, let $k=|E|$.
Note that if
$\{n-k+1,\dots,n\}\cdot \sigma_E=E$ then 
$\{n-k+1,\dots,n\} \cdot \sigma_E \sigma=E \cdot \sigma$
and thus
\[
M^n_{E \cdot \sigma}= M^n_{\{n-k+1,\dots,n\}}\cdot \sigma_E \sigma
= M^n_E\sigma. \qedhere
\]
\end{proof}

Finally, denote by $r_{P_1}$ (resp. $l_{P_1}$) 
the right- (resp. left-) multiplication by $P_1$.
We also consider the operator
$\delta$ on word quasi-symmetric functions,
defined as follows:
\begin{itemize}
    \item if, for some packed word $v$,
        $u=v \cdot (m+1)$, where $m$ is the biggest letter in $v$,
        then $\delta(P_u)=P_v$;
    \item otherwise, $\delta(P_u)=0$.
\end{itemize}
We give a few trivial computational rules relating $r_{P_1}$, $l_{P_1}$,
$\delta$, and the symmetric group action.
In the following equation, $w$ is an element of $\WQSym_n$ or $\WQSym_{n-1}$
and $\sigma$ a permutation of $\SG_{n-1}$.
Then $\overline{\sigma}$ denotes the trivial extension of $\sigma$
to $\{1,\dots,n\}$ ($\overline{\sigma}(n)=n$)
and $\vec{\sigma}$ the permutation of $\SG_n$ defined
by $\vec{\sigma}(1)=1$
and $\vec{\sigma}(i+1)=\sigma(i)+1$
for $i$ in $\{1,\dots,n-1\}$. We have:
\begin{align}
    \delta(w  P_1)&= w, \label{EqDeltaRP1}\\
    \text{if $\deg(w)=n-1>0$ ; }
    \delta(P_1  w)&= P_1  \delta(w) ; \label{EqDeltaLP1}\\
    \delta \big( w\cdot \overline{\sigma} \big)&= \big( \delta(w) \big) \cdot \sigma;
    \label{EqSigmaDelta}\\
    P_1 (w \cdot \sigma) &= (P_1 w) \cdot \vec{\sigma};
    \label{EqSigmaLP1}\\
    (w \cdot \sigma) P_1 &=  (w P_1) \cdot \overline{\sigma}.\label{EqSigmaRP1}
\end{align}

\begin{lemma}
    We have the following compatibility properties between spaces $M^{n}_E$
    and operators $\delta$ and $r_{P_1}$.
    \begin{itemize}
        \item If $E$ is a subset of $\{1,\dots,n-1\}$, then
            \begin{align*}
            r_{P_1}(M_E^{n-1}) &= M_{E \sqcup \{n\}}^n;\\
            l_{P_1}(M_E^{n-1}) &= M_{\{1\} \sqcup \vec{E}}^n,
            \end{align*}
            where $\vec{E}=\{i+1,\ i \in E\}$.
        \item Let $r<n$ be a non-negative integer.
            If $E \nsubseteq \{n-r+1,\dots,n\}$, then
            \[\delta^r(M_E^n) \subset \KK_{n-r}.\]
    \end{itemize}
    \label{LemCompatibilityMDeltaR}
\end{lemma}
\begin{proof}
The first item follows directly from the relevant definitions
and Equations \eqref{EqSigmaLP1} and \eqref{EqSigmaRP1}.

Let us prove the second item.
We begin by the case where $1 \in E$,
that is $E=1 \sqcup \vec{E'}$ for some set $E'$.
Using the first item, $M^n_E=l_{P_1}(M^{n-1}_{E'})$.
In particular, every element $m_E$ of $M^n_E$ can be written as $P_1 \cdot w_{n-1}$
for some element $w \in \WQSym_{n-1}$.
Then, using \eqref{EqDeltaLP1},
\[\delta^r(m_E)=\delta^r(P_1 w_{n-1}) = P_1 \delta^r(w_{n-1}) 
=  \big( \delta^r(w_{n-1}) P_1 \big) \cdot \sigma, \]
where $\sigma \in \SG_{n-r}$
has word notation $23\dots(n-r)1$.
This shows that $\delta^r(m_E)$ belongs to $\KK_{n-r}$, as required.

Now, let $E$ be some set not contained in $\{n\!-\!r\!+\!1,\dots,n\}$.
Choose a permutation $\tau$ in $S_n$ fixing $n\!-\!r\!+\!1,\dots,n$ and
such that $E \cdot \tau$ contains $1$.
Denote by $\check{\tau} \in \SG_{n-r}$ its restriction to $\{1,\dots,n-r\}$.
Iterating \eqref{EqSigmaDelta}, we get that, for $m_E$ in $M_E^n$,
\[\delta^r(m_E) = \big[ \delta^r \big( m_E \cdot \tau\big)\big] \cdot \check{\tau}^{-1} .\]
But $m_E\tau \in M_{\tau(E)}^n$ and hence, using the case considered above,
$\delta^r \big( m_E\tau\big)$ is in $\KK_{n-r}$.
As $\KK_{n-r}$ is stable by the action of $\SG_{n-r}$,
$\delta^r(m_E)$ is also in $\KK_{n-r}$, which is the desired result.
\end{proof}
    
\begin{proposition}
With the notation above,
\[\WQSym_n = \bigoplus_{E \subseteq \{1,\dots,n\}} M^{n}_E.\]
    \label{PropKDirectSum}
\end{proposition}
\begin{proof}
Let us first prove that the sum is indeed a direct sum.
Consider a vanishing linear combination
$\sum_{E \subseteq \{1,\dots,n\}} m_E=0$
with $m_E=u_E P_1^{|E|} \cdot \sigma_E\in M^{n}_E$ (here, $u_E$ lies in $U_{n-|E|}$).
First, write 
\[m_{\emptyset}
  =-\sum_{\gf{E \subseteq \{1,\dots,n\}}{E \neq \emptyset}}
     m_E\]
The left-hand side is in $\U_n$ by definition, while the right-hand side is in
$\KK_n$ (case $r=0$ of the second item of Lemma
\ref{LemCompatibilityMDeltaR}).
Thus $m_{\emptyset}$ necessarily vanishes.

Note that $\delta(m_{\{n\}})=u_{\{n\}}$ lies in $\U_{n-1}$.
But by linearity of $\delta$,
\[\delta(m_{\{n\}})=
  -\sum_{\gf{E \subseteq \{1,\dots,n\}}{E \neq \emptyset,\{n\} }}
   \delta(m_E). \]
But, using the second item of Lemma \ref{LemCompatibilityMDeltaR} for $r=1$,
we get that $\delta(m_E)$ lies in $\KK_{n-1}$ for $E \neq \emptyset,\{n\}$.
Hence, $\delta(m_{\{n\}})=u_{\{n\}}$ lies also in $\KK_{n-1}$ and vanishes.
The same arguments applied to $\sum_E m_E\cdot (i,n)$ imply that
$m_{\{i\}}=0$ (here, $i$ is an positive integer smaller than $n$ and $(i,n)$
is the transposition in $\SG_n$ exchanging $i$ and $n$).

Then we write that
\[\delta^2(m_{\{n-1,n\}})
  = -\sum_{\gf{E \subseteq \{1,\dots,n\}}{|E| \ge 2, E \neq \{n-1,n\} }}
\delta^2(m_E).\]
The condition $|E| \ge 2$ arises because we have already proved that $m_E=0$
for $|E| \le 1$.
Using the second item of Lemma \ref{LemCompatibilityMDeltaR} for $r=2$,
we immediately see that the right-hand side lies in $\KK_{n-2}$.
But $\delta^2(m_{\{n-1,n\}})=u_{\{n-1,n\}}$ also lies in $U_{n-2}$
and thus vanishes.

Applying the same arguments to $\sum_E m_E \cdot \sigma$ for well-chosen permutations
$\sigma$, we prove that $m_E=0$ for any pair $E$ included in $\{1,\dots,n\}$.
    
By iterating the same arguments, we conclude that $m_E$ vanishes for any
subset $E$ of $\{1,\dots,n\}$. This proves that the sum in the statement is
direct.  Let us prove that it is indeed $\WQSym_n$.

We proceed by induction on $n$.
For $n=0$, $\WQSym_0 \simeq \C$ while $\KK_0=\{0\}$.
Hence $U_0\simeq \C$ and $M_\emptyset^0=U_0 \simeq \C$.
Therefore, the result is true in this case.

Assume that it is true for $n-1$.
Consider an element $w_{n-1} P_1$, where $w_{n-1}$ lies in $\WQSym_{n-1}$.
By induction hypothesis
\[w_{n-1}= \sum_{E \subset \{1,\dots,n-1\}} m'_E, \]
where $m'_E \in M^{n-1}_E$.

Using the first item of Lemma \ref{LemCompatibilityMDeltaR},
for any $E \subset \{1,\dots,n-1\}$,
the product $m'_E P_1$ lies in $M^{n}_{E \sqcup \{n\}}$.
This proves that $w_{n-1} P_1$ lies in 
\[\bigoplus_{E \subseteq \{1,\dots,n\}} M^{n}_E.\]
Besides, this sum is invariant by the action of $\SG_n$, because of Lemma
\ref{LemMEWellDefined}.
As it contains elements of the form $w_{n-1} P_1$, it contains $\KK_n$.
Furthermore, it contains $M^{n}_{\emptyset}=U_n$.
Finally, it is equal to $\WQSym_n$.
\end{proof}

Recall that $\dim(\WQSym_n)=OB(n)$ the $n$-th ordered Bell number.
Denote by $\dimKK_n$ the dimension of $\KK_n$.
The spaces $M_E^{n}$ have the same dimension as $U_{n-|E|}$, that is 
$OB(n-|E|) - \dimKK_{n-|E|}$.
Therefore, we have the following immediate numerical corollary
of Proposition \ref{PropKDirectSum}.
\begin{corollary}
    For $n \ge 1$,
    \[OB(n)= \sum_{j=0}^n \binom{n}{j} (OB(n-j) - \dimKK_{n-j}).\]
\end{corollary}
With this relation and the base case $\dimKK_{0}=0$,
the numbers $\dimKK_j$ can be computed inductively.
The first values are $0$, $1$, $1$, $7$, $37$, $271$.
This sequence appears in the Online Encyclopedia of Integer Sequences
\cite[A089677]{Sloane}.
We shall give a simple combinatorial interpretation of it,
not mentioned in \cite{Sloane}.

Recall that $OB(n)$ counts the set-compositions 
({\em i.e.}, ordered set partitions) of $\{1,\dots,n\}$.
Denote by $OB_{\text{odd}}(n)$ the number of set compositions with an odd
number of parts.
A set-composition with an odd number of parts can be specified as follows
\begin{itemize}
\item the first set $I_1$ of the set composition;
\item a set composition of $\{1,\dots,n\} \setminus I_1$ in an even number of
parts.
\end{itemize}
For a given $j$, there are $\binom{n}{j}$ sets $I_1$ of size $j$, and,
for each of them, there are exactly
$OB(n-j) - OB_{\text{odd}}(n-j)$ set compositions with an
even number of parts of $\{1,\dots,n\} \setminus I_1$.
Thus, the sequence $(OB_{\text{odd}}(n))_{n\ge 0}$
satisfies the following induction
\[OB_{\text{odd}}(n) = \sum_{h=1}^n \binom{n}{h} 
(OB(n-h) - OB_{\text{odd}}(n-h)),\]
together with base case $OB_{\text{odd}}(0)=0$.
It is the same induction and base case as for $\dimKK_n$.
Hence one has the following combinatorial interpretation for $\dimKK_n$
\begin{proposition}
    $\dimKK_n$ counts the set-compositions of $\{1,\dots,n\}$
    with an odd number of parts.
\end{proposition}
\begin{corollary}
    \label{cor:DimKerBig}
    The dimension of the kernel of $\Phiab$ in degree $n$
    is at least the number of set-compositions of $\{1,\dots,n\}$ with an odd
    number of parts.
\end{corollary}
\begin{OpenProblem}
    Find a basis of $\KK_n$ indexed by set-compositions of $\{1,\dots,n\}$
        with an odd number of parts.
\end{OpenProblem}
\begin{remark}{\rm
    One also has the remarkable relation
    \[\dimKK_n=\frac{1}{2}(OB(n) - (-1)^n).\]
    This identity, suggested by V. Jovovic \cite[A089677]{Sloane},
    can be proved easily from our combinatorial interpretation,
    but is not useful in this paper.
}
\end{remark}

\subsection{Functions $N_G$ in noncommuting variables}
\label{SubsectNonComNG}
In the next Section, we shall exhibit an explicit independent homogeneous family in $\SolNC'$.
As $\Phiab$ is surjective, this will give us, for each $n \ge 1$, a lower bound on the dimension
of the image of $\Phiab$ in degree $n$.
This will be done by lifting the construction of $N_G$, done in Section~\ref{SubsectNG},
to the noncommutative world.

Let us first lift the one-alphabet function $F_G$.
Take as data a {\em labelled} bipartite graph $\bm{G}$
(with two types of edges) with vertex set
$V= V_1 \sqcup V_2 = \{ 1, \dots, n \}$.
Then we define the noncommutative analog $\bm{F_G}$ of $F_G$ as follows:
\[\bm{F_G}(a_1,a_2,\dots)
  =\sum_{\gf{r:V \to \N}{\text{with order condition}}}
a_{r(1)} a_{r(2)} \dots a_{r(n)}.\]
Here, as usual, the $a_i$ are noncommuting variables.
Clearly, $\bm{F_G}$ is a word quasi-symmetric function.

In the same way, we can define a noncommutative analog $\bm{N_G}$ of $N_G$:
\[\bm{N_G}\left( \begin{array}{ccc}
            b_1 & b_2 & \dots\\
                d_1 & d_2 & \dots
            \end{array} \right)=
\sum_{\gf{r:V \to \N}{\text{with order condition}}}
\bd_{r(1)} \bd_{r(2)} \dots \bd_{r(n)},\]
where we use the abusive 
shorthand notation $\bd_{r(i)}=b_{r(i)}$ for $i \in V_1$
and $\bd_{r(i)}=d_{r(i)}$ for $i \in V_2$.

\begin{example}\label{ex:NG}{\rm
    Consider the graph $\bm{G_\ex}$ drawn on Figure \ref{fig:example_labelled_graph}.
    This is a labelled version of the graph $G_\ex$ of Figure \ref{fig:example_unlabelled_graph}.
    As on this figure, vertices in $V_1$ (resp. $V_2$) are drawn in white (resp. black).
    Edges in $E_{1,2}$ (resp. $E_{2,1}$) are represented directed from their extremity in 
    $V_1$ to their extremity in $V_2$ (resp. from their extremity in $V_2$
    to their extremity in $V_1$).

    \begin{figure}
        \[\begin{tikzpicture}[font=\scriptsize]
            \tikzstyle{bv}=[circle,fill=black,inner sep=0.1mm,minimum size=2mm]
            \tikzstyle{wv}=[circle,draw=black,inner sep=0.1mm,minimum size=2mm]
            \node[wv] (v1) at (0,0) {$2$};
            \node[wv] (v2) at (0,-1) {$3$};
            \node[bv] (v3) at (1,0) {\color{white} $1$};
            \node[bv] (v4) at (1,-1) {\color{white} $5$};
            \node[wv] (v5) at (2,-.5) {$6$};
            \node[bv] (v6) at (3,-.5) {\color{white} $4$};
            \draw[->] (v1) -- (v3);
            \draw[->] (v1) -- (v4);
            \draw[->] (v2) -- (v3);
            \draw[->] (v2) -- (v4);
            \draw[->] (v3) -- (v5);
            \draw[->] (v4) -- (v5);
            \draw[->] (v5) -- (v6);
        \end{tikzpicture}\]
        \caption{Example of a labelled bipartite graph $G$.}
        \label{fig:example_labelled_graph}
    \end{figure}
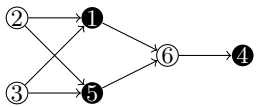

    Let $r$ be a function from its vertex set to $\N$.
    Define 
    \[e:=r(2),\ f:=r(3),\ g:=r(1),\ h:=r(5),\ i:=r(6),\ j:=r(4).\]
    Then, by definition, $r$ satisfy the order condition if and only if
    \[e,f \le g,h < i \le j,\]
so one has,
\[\bm{N_{G_\ex}}\left( \begin{array}{ccc}
        b_1 & b_2 & \dots\\       
            d_1 & d_2 & \dots         
        \end{array} \right) = \sum_{e,f \le g,h < i \le j}
        d_g \, b_e\, b_f\, d_j\, d_h\, b_i,\]
 which is a noncommutative version of the function $N_{G_\ex}$
 given in Example \ref{ex:NG}.
    }
\end{example}
The following extension of Lemma \ref{LemNGInSol} has the very same proof.
\begin{lemma}
Let $G$ be bipartite graph as above.
Assume that each element in $V$ is the extremity
of at least one edge in $E_{1,2}$.
Then the function $\bm{N_G}$ belongs to $\SolNC'$.
    \label{LemNGInSolNC}
\end{lemma}
Proposition \ref{PropFormulaNG} can also be lifted directly.
\begin{proposition}
    For any labelled bipartite graph $G$ with vertex set 
    $V=V_1 \sqcup V_2= \{1,\dots,n\}$,
    \[ \bm{N_G}\left( \begin{array}{ccc}
            b_1 & b_2 & \dots\\
                d_1 & d_2 & \dots
            \end{array} \right) = (-1)^{|V_1|} \Phiab\big(\bm{F_G}(\A)\big).\]
\end{proposition}



\subsection{A large linearly independent explicit family in $(\SolNC')_n$}
\label{subsec:explicit_family}

We shall now construct a large enough family of labelled graphs $\bm{G}$,
such that the corresponding $\bm{N_G}$ are linearly independent.

Consider a set composition $\bm{K}$ of $n$ with an even number of parts 
\[\bm{K}=(K_1,\dots,K_{2\ell}) \text{ with }K_1 \sqcup \dots \sqcup K_{2\ell} 
= \{1,\dots,n\}\]
Then we define $\bm{G_K}$ as follows:
\begin{itemize}
    \item Its vertex set $\{1,\dots,n\}$ is split into two parts:
        \begin{align*}
        V_1&=K_1 \sqcup K_3 \sqcup \dots \sqcup K_{2\ell-1} ;\\
        V_2&=K_2 \sqcup K_4 \sqcup \dots \sqcup K_{2\ell}.
        \end{align*}
    \item Its edge sets are given by
        \begin{align*}
            E_{1,2}= (K_1 \times K_2) \sqcup (K_3 \times K_4) \sqcup \dots \sqcup
                (K_{2\ell-1} \times K_{2\ell}) ;\\
            E_{2,1}= (K_2 \times K_3) \sqcup (K_4 \times K_5) \sqcup \dots \sqcup
                (K_{2\ell-2} \times K_{2\ell-1}).
        \end{align*}
\end{itemize}

\begin{example}{\rm
    If $\bm{K}=\big( \{2,3\},\{1,5\},\{6\},\{4\} \big)$,
    then $\bm{G_K}$ is the graph of figure \ref{fig:example_labelled_graph}.
    }
\end{example}

By Lemma \ref{LemNGInSolNC}, the associated functions $\bm{N_{G_K}}$
belong to $\SolNC'$.
\begin{lemma}
    The functions $\bm{N_{G_K}}$, where $\bm{K}$ runs over set compositions
    of $\{1,\dots,n\}$ (for $n \ge 1$)
    with an even number of parts are linearly independent.
\end{lemma}
\begin{proof}
    With a noncommutative monomial (a word) in $b_i$ and $d_i$,
    we can associate its {\em evaluation}, which we define as
    the integer sequence
    \[(\text{number of }b_1,\text{number of }d_1, \text{number of }b_2,\dots).\]
    It is immediate to see that the monomial in $\bm{N_{G_K}}$
    with the lexicographically largest evaluation is obtained as follows:
    it has letters $b_1$ in positions given by $K_1$, letters $d_1$
    in position given by $K_2$, letters $b_2$ in positions given by $K_3$,\dots
    It follows that the set-composition $\bm{K}$ can be recovered from
    the monomial of lexicographically largest evaluation in $\bm{N_{G_K}}$,
    which implies the linear independence of the $\bm{N_{G_K}}$.
\end{proof}

\begin{corollary}
    \label{cor:DimImBig}
    The dimension of $(\SolNC')_n$, that is of the image of $\Phiab$ in degree $n$
    is at least the number of set-compositions of $\{1,\dots,n\}$ with an even
    number of parts.
\end{corollary}

\begin{remark}\label{rmq:Luoto}{\rm
    The corresponding family $\bm{F_{G_K}}$ is a natural noncommutative lifting
    of the basis introduced by K. Luoto in \cite{Luoto}
    (here, we only lift elements indexed by even-length composition,
    but it would not be hard to lift all elements).

    An interesting feature is that the above argument together with the fact that
    $\bm{N_{G_K}} = (-1)^{|V_1|} \Phiab\big(\bm{F_{G_K}}\big) $ implies
    that the $\bm{F_{G_K}}$ are linearly independent,
    which would have been difficult to prove directly.
}
\end{remark}

\subsection{Conclusion}
As the dimension of $(\SolNC)_n \simeq \WQSym_n$ is the number of set-compositions
of $[n]$, Corollaries~\ref{cor:DimKerBig} and \ref{cor:DimImBig} imply:
\begin{theorem}\label{thm:Ker_NC}
        {\rm (i)} The kernel of $\Phiab$ is $\bigoplus_{n\geq 1} \KK_n$
            and its degree $n$ component has a dimension equal
            to the number of set-compositions of $\{1,\dots,n\}$ 
            in an odd number of parts.
            Besides, it is the smallest homogeneous two-sided ideal containing $P_1$ and 
            whose homogeneous components are stable by the action of the symmetric groups.

        {\rm (ii)} The dimension of the degree $n$ component of the image of $\Phiab$,
            that is of $(\SolNC')_n$,
            is exactly the number of set-compositions of $\{1,\dots,n\}$
            in an even number of parts.
\end{theorem}
\section{Quasi-symmetric functions on Young diagrams in terms of other sets of coordinates}
\label{SectFuture}

In this section, we discuss some properties of our algebra $Q\Lambda$ of functions
on Young diagrams, seen in terms of other sets of coordinates.
These are mainly open problems and directions for future research.

\subsection{Row coordinates of the Young diagram}
The most natural way to describe a Young diagram is by its row coordinates, that is
the parts of the corresponding partition $\lambda_1,\lambda_2,\dots$
It is shown in \cite[top of page 9]{IO02} that the algebra
of symmetric functions on Young diagrams is the algebra of 
{\em shifted symmetric functions}\footnote{The shifted symmetric function algebra 
is a deformation of the symmetric function algebra, where the symmetry in $\lambda_1,\lambda_2,\dots$
is replaced by a symmetry in $\lambda_1-1,\lambda_2-2,\dots$.
This algebra has been intensively studied in \cite{OO98} and
displays surprisingly nice properties.}
in the row coordinates $\lambda_1,\lambda_2,\dots$

Looking at Figure \ref{FigRussian}, we easily see that 
our virtual alphabet $\X$ can be written as
\begin{equation}
    \X = \ominus (\lambda_1) \oplus (\lambda_1-1) \ominus (\lambda_2-1)
\oplus (\lambda_2-2) \ominus (\lambda_3-2) \oplus \cdots.
\label{eq:X_RowCoordinates}
\end{equation}
Indeed, after removing consecutive terms with equal values and opposite signs
on the right-hand side, we are left with the definition of $\X$.

Therefore, quasi-symmetric functions evaluated in the right hand side of \eqref{eq:X_RowCoordinates}
is a natural extension of shifted symmetric functions (at least from the point of view
of functions on Young diagrams).
It would be interested to investigate whether some properties of \cite{OO98}
can be extended to this algebra.

\subsection{Frobenius coordinates}
(Modified) Frobenius coordinates of a Young diagram $\lambda$ are defined as follows:
let $d$ be the biggest integer such that $\lambda_d \ge d$, then set
(for $1 \le i \le d$)
\begin{align*}
    a_i &= \lambda_i -i +1/2 \\
    b_i &=\lambda'_i -i +1/2,
\end{align*}
where $\lambda'$ denotes, as usual, the conjugate of $\lambda$.
Ivanov and Olshanski have shown \cite[page 8]{IO02} that 
the algebra $\Lambda$ of symmetric functions on Young diagrams
correspond to super-symmetric functions in Frobenius coordinates,
that is symmetric functions in the virtual alphabet\footnote{Unlike in \eqref{EqVirtualAlphabet},
we use usual $+$ and $-$ signs for addition and soustraction of 
virtual alphabets for symmetric functions,
as these operators commute in this context.}:
\[(a_1) - (-b_1) + (a_2) - (-b_2) + \cdots\]

As for row coordinates, the virtual alphabet $\X$ is easily expressible
in terms of Frobenius coordinates
\begin{multline}
    \X = \ominus (a_1+1/2) \oplus (a_1 -1/2) \ominus (a_2 +1/2) \oplus (a_2 -1/2) \ominus \dots \\
    \dots \oplus (b_2 -1/2) \ominus (b_2 +1/2) \oplus (b_1 -1/2) \ominus (b_1+1/2).
\label{eq:X_Frobenius}
\end{multline}
Thus quasi-symmetric functions in terms of this alphabet
gives a quasi-symmetric analog of super-symmetric function,
which is natural from a ``functions on Young diagrams'' point of view.
This could be interested to investigate.

\subsection{Contents}
The multiset of contents of a Young diagram $\lambda$ is defined as the multiset
$\{j-i: (i,j) \in \lambda\}$.
The algebra $\Lambda$ of symmetric functions on Young diagrams
correspond to symmetric functions in the multiset of contents with coefficients in $\C[|\lambda|]$
({\em i.e.} coefficients may depend polynomially on the size $|\la|$ of the partition),
see \cite{CGS04} or \cite[Remark 1]{Fer12}.

Unfortunately, we have not found a way to order the set of contents and express
$\X$ in terms of contents with a formula similar to Equations~\eqref{eq:X_RowCoordinates}
and \eqref{eq:X_Frobenius}.
Thus, we leave open the following question:
is there a nice description of our algebra $Q\Lambda$
in terms of the multiset of contents?

\section{Appendix}

The definitions of the virtual alphabet $\X$ and of the functions
$\H_I$ can be made more transparent if one introduces the formalism
of noncommutative symmetric functions \cite{NCSF1}.

For a totally ordered alphabet $A=\{a_i|i\ge 1\}$ of noncommuting
variables, and $t$ an indeterminate, one sets
\begin{equation}
\sigma_t(A)=\prod_{i\ge 1}^\rightarrow (1-ta_i)^{-1} =\sum_{n\ge 0} S_n(A)t^n\,.
\end{equation}
The complete symmetric functions $S_n(A)$ generate a free associative algebra
$\Sym(A)$ (noncommutative symmetric functions). One denotes by
\begin{equation}
S^I(A):= S_{i_1}S_{i_2}\cdots S_{i_r}
\end{equation}
its natural basis. Actually, $\Sym(A)$ is a Hopf algebra, and its graded
dual is $\QSym$. This can be deduced from the noncommutative Cauchy
formula
\begin{equation}
\sigma_1(XA):= \prod_{i\ge 1}^\rightarrow \sigma_{x_i}(A) = \sum_I M_I(X)S^I(A)
\end{equation}
which allows to identify $M_I$ with the dual basis of $S^I$ \cite{NCSF1,MR}.

Now, the virtual alphabet $\X$ can be defined by
\begin{equation}
\sigma_1(\X A) = \sum_I M_I(\X)S^I(A) = \prod_{i\ge 1}^\rightarrow \sigma_{x_i}(A)^{(-1)^i}\,.
\end{equation}
It is clear that the right-hand side (seen as a polynomial in infinitely many variables
and coefficients in $\Sym(A)$) is a solution of~\eqref{eqfoncQS}.
Hence all the coefficients in its $S^I$ expansion, that is the $M_I(\X)$,
are also solution of~\eqref{eqfoncQS}.

The set of polynomials in $p_1,p_2,\dots,q_1,q_2,\dots$,
that are solutions of equations~\eqref{EqQZero} and \eqref{EqPZero}
can also been described with the language of virtual alphabets.

Consider the virtual alphabet $\Y$ defined by
\begin{equation}\label{eq:serpq}
\sigma_1(\Y A)=\prod_{i\ge 1}^\rightarrow \sigma_{q'_i}(A)^{\frac{p_i}{q'_i}}\,.
\end{equation}
Then we have:
\begin{theorem}
    A linear basis for the space $\Sol'$ of 
    solutions of~\eqref{EqQZero} and \eqref{EqPZero}
    is given by the function $M_I(\Y)$, where $I$ runs over compositions with
    the last part bigger than $1$.
    \label{Thm:Solpq_Via_LambdaRing}
\end{theorem}
\begin{proof}
    On the series \eqref{eq:serpq}, it is immediate that 
    quasi-symmetric functions in $\Y$ satisfy conditions~\eqref{EqQZero} and
    \eqref{EqPZero}, except equation~\eqref{EqQZero} for $i=m$.
    Indeed, when we restrict to $m$ variables and set $q_m=0$,
    then $q'_m=0$ and
    \[\sigma_{q'_m}(A)^{\frac{p_m}{q'_m}} = \exp(p_m S_1(A)),\]
    and the variable $p_m$ does not disappear.

    This can be corrected as follows: define $\tilde{M}_I$ by
\begin{equation}
    \label{Eq:Def_MTilde}
    \sum_I \tilde{M}_I S^I := \sigma_1(\Y A)e^{-\left(\sum_i p_i\right) S_1}\,.
\end{equation}
Then $\tilde{M}_I$ satisfies equations~\eqref{EqQZero} and \eqref{EqPZero}.
But the second factor in the right-hand side of~\eqref{Eq:Def_MTilde}
contains only $S_1$, thus, if the last part of a composition $I$ is bigger than $1$,
we can forget this factor when extracting the coefficient of $S^I$.
In other terms, for such a composition $I$
\[\tilde{M}_I = M_I(\Y). \]
This shows that the $M_I(\Y)$ are indeed solutions of~\eqref{EqQZero} and \eqref{EqPZero}.
If we substitute $p_i=q'_i=y_i$, we recover $M_I(Y)$, and hence they are linearly
independent.
A dimension argument finishes the proof.
\end{proof}

This implies that $\H_I=F(\Y)$ for some quasi-symmetric function $F$
and we shall now identify $F$.

Let $H_I$ be the function defined in equation~\eqref{HIX}.
Alternatively, let
$\phi_n$ be the noncommutative symmetric functions defined by
\begin{equation}
\log \sigma_t(A) =\sum_{n\ge 1}\phi_n(A)t^n,
\end{equation}
then $H_I\in \QSym$ is the dual basis of the basis $\phi^I$ of $\Sym$.

\begin{proposition}   
$\H_I= H_I(\Y)$.
\end{proposition}
\begin{proof}
    After substitution $p_i=q'_i=y_i$, $F(\Y)$ is sent to $F(Y)$
    and $\H_I$ to $H_I(Y)$ (see Section~\ref{SubSectConfusingAlphabets}).
    This yields $F=H_I$.
\end{proof}

\begin{remark}{\rm
    Equation \eqref{Eq:Def_MTilde} is
    an avatar of the $(1-{\mathbb E})$-transform investigated
in \cite{HLNT}, another example being the so-called quasi-shuffle
regularization of Multiple Zeta Values (see, {\it e.g.}, \cite{Cartier}).
}
\end{remark}


\end{document}